\documentclass[11pt,a4paper, reqno, centertags]{amsart}

\usepackage{indentfirst} 
\usepackage{amsmath, amsthm, amssymb,xfrac, mathrsfs} 
\usepackage{bbm}
\usepackage{BOONDOX-calo, graphicx, accents}
\usepackage{tikz} 
\usetikzlibrary{hobby}
\usetikzlibrary{decorations.pathreplacing}
\usepackage{wrapfig} 
\usepackage{geometry}
\usepackage{xcolor} 
\usepackage{verbatim} 
\usepackage{enumerate} 

\usepackage{hyperref} 
\hypersetup{
colorlinks=true,
citecolor=black!50!red,
linkcolor=black!50!red,
linktoc=all
}
\usepackage[all]{hypcap} 

\newcounter{constant}

\newcounter{bigconstant}

\newcommand{\jap}[1]{\langle #1\rangle}

\newtheorem{theorem}{Theorem}[section]

\newtheorem{lemma}[theorem]{Lemma}

\newtheorem{corollary}[theorem]{Corollary}

\numberwithin{equation}{section} 

\theoremstyle{definition}

\newtheorem{remark}[theorem]{Remark}

\newcommand{\R}{\mathbb{R}}

\newcommand{\N}{\mathbb{N}}



\makeatother

\begin{document}
\title[Nonlinear smoothing implies bounds on radius of analyticity]{Nonlinear smoothing implies improved lower bounds on the radius of spatial analyticity for nonlinear dispersive equations}
    \author[M. Baldasso and S. Correia]{ Mikaela Baldasso and Simão Correia}

\maketitle

\begin{abstract}
We provide a roadmap to establish improved lower bounds on the decay rate of the uniform radius of analyticity $\sigma(T)$ for a given nonlinear dispersive equation, reducing the problem to the derivation of nonlinear smoothing estimates with a specific distribution of extra derivatives. We apply this strategy for both the defocusing generalized KdV and the nonlinear Schrödinger equations with odd pure-power nonlinearity. For both equations, we reach the lower bound $\sigma(T)\gtrsim T^{-\frac{1}{2}-\epsilon}$, for any $\epsilon>0$, thus improving all available results in the current literature.

\end{abstract}


\section{Introduction}\label{s:intro}

\subsection{Setting of the problem and main results} In this work, we consider the initial value problems (IVPs) associated to the defocusing generalized Korteweg-de Vries  equation
\begin{equation}\label{eq:gKdV}\tag{gKdV}
\begin{cases}
\partial_{t}u+\partial_{x}^{3}u-u^{k}\partial_{x}u=0, \, &x\in \R, \, t \in \R,\\
u(x,0)=u_0(x),
\end{cases}
\end{equation}
where $u$ is a real-valued function and $k \geq 4$ is even, and to the nonlinear Schrödinger equation 
\begin{equation}\label{eq:schrodinger}\tag{NLS}
\begin{cases}
i\partial_{t}v+\partial_{x}^{2}v-|v|^{p-1}v=0, \, &x\in \R, \, t \in \R,\\
v(x,0)=v_0(x),
\end{cases}
\end{equation}
where $v$ is a complex-valued function and $p\geq 3$ is odd.


In the last fifteen years, there has been a growing interest in extending the well-posedness theory for nonlinear dispersive equations with given data in the Gevrey spaces $G^{\sigma,s}(\R)$, see for example \cite{BalPan, BonaGruKal, FigPan, GruKal, Hayashi_Analyt, Selberg_DKG, SelDaSilva, SelTes}. For $\sigma>0$ and $s\in \R$, the Gevrey space $G^{\sigma,s}(\R)$ is defined as
\begin{equation}\label{eq:Gevrey1}
    G^{\sigma,\,s}(\R)=\left\lbrace f \in L^2(\R):\, \|f\|_{G^{\sigma,\,s}(\R)}^2=\int_{\R} e^{2\sigma|\xi|}\langle \xi\rangle^{2s}|\widehat{f}(\xi)|^2d\xi <\infty\right\rbrace,
\end{equation}
where $\langle \xi \rangle = \sqrt{1+|\xi|^2}$ and $\widehat{f}$ denotes the spatial Fourier transform of $f$,
\begin{equation*}
    \widehat{f}(\xi)=\frac{1}{\sqrt{2\pi}}\int_{\R}e^{-ix\xi}f(x)dx.
\end{equation*}

According to the Paley-Wiener Theorem (see \cite{katznelson}), a function $f$ belongs to $G^{\sigma,s}(\R)$ if and only if it can be holomorphically extended in the strip $S_{\sigma} =\{x+iy:\, x,\, y\in \R, \, |y|<\sigma\}$ to a function $F$ satisfying $\sup_{|y|<\sigma}\|F(x+iy)\|_{H_x^s}<\infty$. In other terms, the parameter $\sigma>0$ determines the width of the complex strip where the functions in $G^{\sigma,s}(\R)$ can be extended. For this reason, it is often referred to as the uniform radius of analyticity of $f$. In this context, in addition to the local well-posedness in $G^{\sigma,s}(\R)$, an important question concerns the asymptotic behaviour of $\sigma$ when local solutions can be globally extended in time. More precisely, we aim to obtain lower bounds on $\sigma(T)$, the uniform radius of analyticity of the solution at time $T$.

In the context of the generalized KdV equation \eqref{eq:gKdV}, Bona, Gruji{\'c} and Kalisch \cite{BonaGruKal, GruKal}  proved well-posedness results for the IVP \eqref{eq:gKdV} with data in the Gevrey space $G^{\sigma,s}(\R)$ with $s>\frac{3}{2}$. In \cite{BonaGruKal}, the authors also obtained $\sigma(T) \gtrsim T^{-12}$ for $k=1,2$ and $\sigma(T) \gtrsim T^{-(k^2+3k+2)}$  for $k\geq 3$. For $k=1$, using an almost-conservation quantity at the $L^2$-level, Selberg and Silva \cite{SelDaSilva} obtained the bound $\sigma(T)\gtrsim T^{-(4/3+\epsilon)}$ (see also \cite{Tasf_kdv}). This result was further improved to $\sigma(T)\gtrsim T^{-1/4}$ by Huang and Wang in \cite{HuangWang}. In the case $k=2$, Figueira and Panthee \cite{FigPan} derived the bound $\sigma(T)\gtrsim T^{-4/3}$ and in \cite{FigPan_mkdv} this result was improved to $\sigma(T)\gtrsim T^{-1/2}$ by using the $H^2$-level conservation law. For $k=3$, using once again an almost-conserved quantity at the $L^2$-level, the authors in \cite{SelTes} obtained the lower bound $\sigma(T)\gtrsim T^{-2}$. Recently, the first author and Panthee \cite{BalPan} improved the results for $k\ge 4$, proving local well-posedness for $s>\frac{k-4}{2k}$ and reaching $\sigma(T) \gtrsim T^{-\left(\frac{2k}{k+4}+\epsilon\right)}$.

For the \eqref{eq:schrodinger} equation, Tesfahun \cite{tas_nls} established the local well-posedness in $G^{\sigma,1}(\R^d)$ for $p = 3$ and spatial dimensions $d=1,2,3$. Moreover, he extended the local solutions in time obtaining $\sigma(T) \gtrsim T^{-1}$ as a lower bound for the radius of spatial analyticity of the solutions. For the one- and two-dimensional cases, in \cite{AKS}, Ahn, Kim and Seo extended Tesfahun’s result to any odd integer $p > 3$. More recently, Belayneh and Getachew \cite{BirGet} improved these results obtaining sharper lower bounds: $\sigma(T) \gtrsim T^{-\frac{4}{5}}$ for $p=3$ and $\sigma(T) \gtrsim T^{-\frac{2}{3}}$ for $p>3$. For related results in the context of systems of Schrödinger equations with quadratic nonlinearities, see \cite{FNP, Sato}.

A common feature of the referred literature is that each result relies on \textit{ad hoc} combinations of dispersive estimates (such as Kato smoothing, maximal function and Strichartz estimates), with no clear insight on the optimality of the subsequent lower bound on $\sigma(T)$. In this work, we show how the lower bounds on the radius of analyticity can be \emph{systematically} reduced to specific nonlinear smoothing estimates (see Sections \ref{sec:gKdV} and \ref{sec:NLS}). Heuristically, these estimates translate the phenomenon in which the nonlinear part of the evolution is actually smoother than the initial data. This effect has been studied extensively in the past twenty years \cite{COS,  tzirakis2, tzirakis3, tzirakis1, IMOS, KeraaniVargas} and has been a source of progress in other topics in nonlinear dispersive equations, such as dispersive blow-up \cite{BPSS, BonaSaut, BonaSaut2, LinPasSil}, global well-posedness at low regularity \cite{CKSTT2, CKSTT1, C_4kdv, Farah_5kdv,  GPS, MSWX} and even local well-posedness \cite{CLS}.

\medskip

We now state the main results of this work. We begin with the \eqref{eq:gKdV} equation.

\begin{theorem}\label{teo:globalkdv}
    Let $\sigma_0>0$, $k\geq 4$ be even and $u_0 \in G^{\sigma_0,\, 1}(\R)$. Then, for any $T\geq 0$, the local solution\footnote{The existence of local solution is ensured by Theorem \ref{teo:localkdv} (see also \cite{BalPan}).} $u$ to \eqref{eq:gKdV} extends to the time interval $[-T,T]$ and satisfies
    \begin{equation*}
        u\in C([-T,\,T]:G^{\sigma(T),\, 1}(\R)), \quad \text{with } \sigma(T) \geq \min \left\lbrace\sigma_0 , \, cT^{-\left(\frac{1}{2}+\epsilon \right)}\right\rbrace,
    \end{equation*}
where $\epsilon>0$ is arbitrarily small and $c$ is a positive constant depending on $k$,  $\sigma_0$, $\epsilon$ and $\|u_0\|_{G^{\sigma,\, 1}}$.
\end{theorem}

\begin{remark}
As we are dealing with the case $k\ge 4$, the space $L^2(\R)$ is scaling-(super)critical and thus the mass conservation law is not appropriate to prove lower bounds on $\sigma(T)$, This is in contrast with the cases $k=1,3$, where one can use the conservation of mass directly \cite{HuangWang, SelDaSilva, SelTes, Tasf_kdv}. It would be interesting to check whether our methodology, when applied to these cases, yields improved lower bounds.
\end{remark}
\begin{remark}
    For the modified KdV ($k=2$), Figueira and Panthee \cite{FigPan_mkdv} were able to reach $\sigma(T)\ge T^{-\frac 12}$ by using an $H^2$-conservation law, which is unavailable for $k\ge 4$. Applying our arguments, which rely only on the $H^1$-conservation law, one can reach (almost) the same decay estimate. This indicates that the application of our strategy at the $H^2$-level may lead to substantial improvements on bounds for the decay rate for the modified KdV equation.
\end{remark}

For the \eqref{eq:schrodinger} equation, we prove an analogous result.

\begin{theorem}\label{teo:globalsch}
    Let $\sigma_0>0$, $p\geq 3$ be odd and  $u_0 \in G^{\sigma_0,\, 1}(\R)$. Then, for any $T\geq 0$, the local solution\footnote{The existence of local solution is ensured by Theorem \ref{teo:localsch} (see also \cite{AKS}).} $u$ to \eqref{eq:schrodinger} extends to the time interval $[-T,T]$ and satisfies
    \begin{equation*}
        u\in C([-T,\,T];G^{\sigma(T),\, 1}(\R)), \quad \text{with } \sigma(T) \geq \min \left\lbrace\sigma_0 , \, cT^{-\left(\frac{1}{2}+\epsilon \right)}\right\rbrace,
    \end{equation*}
where $\epsilon>0$ is arbitrarily small and $c$ is a positive constant depending on $p$,  $\sigma_0$, $\epsilon$ and $\|u_0\|_{G^{\sigma,\, 1}}$.
\end{theorem}

The results above represent a considerable upgrade when compared with the aforementioned literature, as we improve \emph{all} known lower bounds on the radius of analyticity. We expect that the arguments presented in this work can be extended to other dispersive models, yielding improved bounds on the uniform radius of analyticity.

\bigskip

\subsection{Brief description of the method.} Having stated the main results, let us now give a brief overview of the methodology to lower bound the quantity $\sigma(T)$ (see, for example, \cite{AKS, BalPan,  tas_nls, Tasf_kdv}).  As is well-known, real solutions of \eqref{eq:gKdV} enjoy the conservation of mass
\begin{equation}\label{eq:masskdv}
    M(u)=\int_{\R} |u|^2dx,
\end{equation}
and also the conservation of energy
\begin{equation}\label{eq:energykdv}
    E_{\text{KdV}}(u)=\int_{\R}(\partial_x u)^2dx +\frac{2}{(k+1)(k+2)}\int_{\R}u^{k+2}dx.
\end{equation}
Similarly, solutions to \eqref{eq:schrodinger} conserve the mass  \eqref{eq:masskdv} and
\begin{equation}\label{eq:energysch}
    E_{\text{NLS}}(v)=\int_{\R}|\partial_x v|^2dx +\frac{2}{p+1}\int_{\R}|v|^{p+1}dx.
\end{equation}
By standard arguments, this is enough to prove global well-posedness over $H^1(\R)$, with uniform bounds. As $H^1(\R)=G^{0,1}(\R)$, one may expect that, for $\sigma>0$ small, the $G^{\sigma,1}(\R)$ norm cannot grow too rapidly. To make this precise, one defines a modified energy $E_\sigma$ which controls the $G^{\sigma,1}(\R)$ norm and which satisfies an almost conservation law  of the form
\begin{equation}\label{eq:ACLexponential}
\sup_{t\in[0,T']} E_\sigma(t)\leq E_\sigma(0)+C\sigma^\theta E_\sigma(0)^a(1+E_\sigma(0)^b)
\end{equation} 
over the local existence time interval, for some $\theta\in [0,1]$ and $a,b>0$. This estimate is then used to extend the local solution globally in time by decomposing a given time interval $[0,T]$ into short subintervals and iteratively applying the local result over $G^{\sigma, 1}(\R)$. During this iterative process, in order to compensate for the possible growth of $E_\sigma$, one must suitably shrink $\sigma$. This provides a lower bound of the form $\sigma(T)\geq cT^{-\frac{1}{\theta}}$.

\begin{remark}
    This technique can be traced back to the I-method developed by Tao \emph{at al.} \cite{CKSTT2, CKSTT1} (which is a refinement of the high-low decomposition method of Bourgain \cite{Bourg_nls_book}). In this context, the goal is to use the energy conservation to prove global well-posedness over $H^s(\R)$, with $s$ slightly smaller than one. To that end, one introduces the operator $I_N$ which cuts off high frequencies,
\begin{equation}\label{eq:multi_imethod}
        \widehat{I_Nu}(\xi)=m_N(\xi)\widehat{u}(\xi),\quad m_N(\xi)=\begin{cases}
        1,&|\xi|<N\\\left(\frac{N}{|\xi|}\right)^{1-s},& |\xi|>N
    \end{cases},
\end{equation}
    so that $u\in H^s(\R)$ if and only if $I_Nu\in H^1(\R)$. After showing an almost conservation law of the form
    $$
    E(I_Nu(t)) \lesssim E(I_Nu(0)) + N^{-\frac 1\theta}E(I_Nu(0))^a,
    $$
    an iterative argument as described above yields an $H^s(\R)$ bound over any time interval, for $s$ depending on the parameter $\theta$ (which once again measures the proximity to the actual conservation law).
\end{remark}

Our goal is then to prove that \eqref{eq:ACLexponential} holds for\footnote{Given $a\in \R$, $a^+$ (resp. $a^-$) denotes a number slightly larger (resp. smaller) than $a$.} $\theta=2^-$. To that end, we introduce two key ingredients:

\emph{1. Choice of an appropriate Gevrey weight.} 
The proof of inequality \eqref{eq:ACLexponential} (see, for example, \cite{AKS, BalPan, Tasf_kdv}) usually relies on the following elementary estimate for the exponential function 
\begin{equation}
e^{\sigma|\xi|}-1\leq (\sigma|\xi|)^\theta e^{\sigma|\xi|}, \quad \theta \in [0,1]. 
\end{equation}
To prove a refined almost conservation law for \eqref{eq:schrodinger}, the authors in \cite{FigPan_mkdv} used a modified Gevrey space by replacing the exponential weight $e^{\sigma|\xi|}$ with the weight $\cosh(\sigma|\xi|)$. The advantage of this equivalent weight lies in the fact that the hyperbolic cosine satisfies the bound
\begin{equation}
\cosh(\sigma|\xi|)-1\leq (\sigma|\xi|)^\theta \cosh(\sigma|\xi|), \quad \theta \in [0,2]. 
\end{equation}
Intuitively, the fact that the first derivative of $\cosh(\sigma|\xi|)$ vanishes at the origin makes it possible to achieve a quadratic power in $\sigma$. This observation suggests that, in order to construct an almost conserved quantity with better decay properties, one may consider replacing the exponential weight by an equivalent function whose derivatives vanish near the origin. To this end, inspired by \eqref{eq:multi_imethod}, we introduce a smooth function defined by
\begin{eqnarray}\label{eq:operator}
m(\xi)=
\begin{cases}
1, \quad  &|\xi|\leq 1,\\
e^{|\xi|},\quad  &|\xi|\geq 2.
\end{cases}
\end{eqnarray}
and for $\sigma>0$, we define the scaled version
\begin{equation}\label{eq:operatorsigma}
m_\sigma(\xi)=m(\sigma\xi).
\end{equation}
It is straightforward to verify that $m_\sigma(\xi)\simeq e^{\sigma|\xi|}$ and thus we may consider the equivalent Gevrey norm given by
\begin{equation}\label{eq:Gevrey2}
\|f\|_{G^{\sigma,s}}=\left(\int_{\R} m_\sigma^2(\xi)\langle\xi\rangle^{2s}|\widehat{f}(\xi)|^2d\xi\right)^\frac{1}{2}.
\end{equation}
This new weight removes technical restrictions on $\theta$ in \eqref{eq:ACLexponential} (caused by an inappropriate functional setting) and allows us to truly relate $\theta$ with the dispersive effects of the underlying equation. 

\emph{2. Nonlinear smoothing estimates.} In order to prove \eqref{eq:ACLexponential}, we must control the time average of $\frac{d}{dt}E_\sigma(t)$. In the \eqref{eq:gKdV} case, using the equation and the properties of $m_\sigma$, the task is reduced to proving the estimates
\begin{equation}\label{eq:aim1_intro}
\|\partial_x\big(|D_x|^{\theta-1}w\cdot |D_x| w\cdot w^{k-1}\big)\|_{X^{1,b-1}}\lesssim \|w\|_{X^{1,b}}^{k+1}
\end{equation}
and
\begin{equation}\label{eq:aim2_intro}
\||D_x|^{\theta}w\cdot |D_x| w\cdot w^{2k-1}\|_{X^{-1,b-1}}\lesssim \|w\|_{X^{1,b}}^{2k+1}
\end{equation}
(see Section \ref{sec:prelim} for a precise definition of the spaces $X^{s,b}$).
Let us focus on the first estimate. Assuming that we may interchange derivatives, the estimate is akin to
\begin{equation}
    \|\partial_x(w^{k+1})\|_{X^{1+\theta,b-1}}\lesssim \|w\|_{X^{1,b}}^{k+1}.
\end{equation}
This is precisely a nonlinear smoothing estimate for \eqref{eq:gKdV}: one controls the nonlinear term in $H^{1+\theta}(\R)$ by the $H^1(\R)$ norm of the solution. In \cite{COS}, the second author, together with Oliveira and Silva, proposed a general method to prove nonlinear smoothing estimates through the so-called frequency-restricted estimates. In the particular cases of \eqref{eq:gKdV} and \eqref{eq:schrodinger}, the theory predicts that the estimate holds for $\theta<1$. If applied directly, this would only yield $\sigma(T)\gtrsim T^{-1-\epsilon}$. However, in the specific case of \eqref{eq:aim1_intro}-\eqref{eq:aim2_intro}, the fact that the extra derivatives are distributed between different $w$'s allows us to actually reach $\theta<2$, which justifies the claimed lower bound $\sigma(T)\gtrsim T^{-\frac12-\epsilon}$.

\bigskip

The remainder of this work is organized as follows. In Section \ref{sec:prelim}, we introduce the functional setting and prove a key estimate on the weight $m_\sigma$. In Sections \ref{sec:gKdV} and \ref{sec:NLS}, we reduce the proofs of Theorems \ref{teo:globalkdv} and \ref{teo:globalsch} to the derivation of particular nonlinear smoothing estimates. Finally, in Section \ref{sec:multilinear}, we prove the multilinear versions of such bounds using the frequency-restricted estimates introduced in \cite{COS}.

\bigskip
\noindent
\textbf{Acknowledgements.} M.B. was partially supported by CAPES grant 88881.982713/2024-01 and is grateful to Instituto Superior Técnico – Universidade de Lisboa for its hospitality during the development of this work. S.C. was partially supported by the Portuguese government
through FCT - Fundação para a Ciência e a Tecnologia, I.P., project UIDB/04459/2020 with DOI identifier
10-54499/UIDP/04459/2020 (CAMGSD).

\section{Preliminaries}\label{sec:prelim}
In order to state the necessary local well-posedness results, we recall the functional setting used in \cite{AKS, BalPan}, which is based on the Bourgain spaces introduced in \cite{Bourg1, Bourg2}. Given $\sigma\geq 0$ and $s,\, b \in \R$, we define the Gevrey-Bourgain space, denoted by $X^{\sigma,\, s,\,b}(\R^2)$, with the norm
\begin{equation*}
\|w\|_{X^{\sigma,\, s,\,b}}=\|m_\sigma(\xi)\langle \xi \rangle^{s} \langle \tau - L(\xi) \rangle ^{b} \tilde{w}(\xi, \tau)\|_{L^{2}_{\tau,\, \xi}},
\end{equation*} 
where 
\begin{eqnarray}\label{eq:phase}
L(\xi)=
\begin{cases}
\xi^3, \, &\text{for the gKdV equation,}\\
-\xi^2, &\text{for the Schrödinger equation.}
\end{cases}
\end{eqnarray} For $\sigma=0$, we recover the classical Bourgain space  $X^{s,\,b}(\R^2)$ with the norm given by
\begin{equation*}
\|w\|_{X^{s,\,b}}=\|\langle \xi \rangle^{s} \langle \tau - L(\xi)\rangle ^{b} \tilde{w}(\xi, \tau)\|_{L^{2}_{\tau,\, \xi}}.
\end{equation*} 
These spaces are related through the equality
\begin{equation}\label{eq:exponential}
\|m_\sigma (D_x)u\|_{X^{s,\, b}}=\|u\|_{X^{\sigma,\,s,\, b}},\qquad \text{where}\quad\widehat{m_\sigma (D_x)u}(\xi)=m_\sigma (\xi)\widehat{u}(\xi).
\end{equation}
\begin{remark}
    The usual definition involves the weight $e^{\sigma|\xi|}$ instead of $m_\sigma$. However, since these weights are equivalent, this modification will have no effect in the existing local well-posedness results.
\end{remark}

For $T>0$, the restrictions of $X^{ s,\,b}(\R^2)$ and $X^{\sigma, \, s,\,b}(\R^2)$ to a time slab $\R \times (-T,\, T)$, denoted by $X_{T}^{ s,\,b}(\R^2)$ and $X_{T}^{\sigma, \, s,\,b}(\R^2)$, respectively, are  equipped with the norms
\begin{align*}
\|u\|_{X_{T}^{s,\,b}}&=\text{inf}\{\|v\|_{X^{s,\,b}}: v=u \text{ on } \R \times (-T,\, T)\},\\
\|u\|_{X_{T}^{\sigma, \,s,\,b}}&=\text{inf}\{\|v\|_{X^{\sigma, \, s,\,b}}: v=u \text{ on } \R \times (-T,\, T)\}.
\end{align*}
These spaces are well-suited for the construction of local-in-time solutions in the Gevrey space (see Subsections \ref{subsec:kdv} and \ref{subsec:nls}). For the present work, we need a technical lemma, whose proof follows from the arguments in \cite[Proof of Lemma 7]{Selberg_DKG}.
\begin{lemma}\label{Lemma2.2}\label{lemma:restriction} Let $\sigma \geq 0$, $s \in \R$,  $-\frac{1}{2}<\tilde{b}<\frac{1}{2}$ and $T>0$. Then, for any time interval $I \subset [-T, \, T]$, we have 
\begin{equation*}
\|\chi_{I}u\|_{X^{\sigma,\,s,\, \tilde{b}}} \leq C \|u\|_{X_{T}^{\sigma,\,s,\, \tilde{b}}},
\end{equation*}
where $\chi_{I}$ denotes  characteristic function of $I$ and $C>0$ depends only on $\tilde{b}$.
\end{lemma}

\bigskip
We conclude this section with a key estimate involving the weight $m_\sigma$.

\begin{lemma}\label{lemma:m} Let $\xi=\sum_{j=1}^{k}\xi_j$ for $\xi_j\in\R$ such that $|\xi_k|\leq \dots \leq |\xi_1|$, where $k\geq 2$ is an integer. Then, for $\theta\geq 1$,
\begin{equation}\label{eq:estm}
\Bigg|1-m_\sigma(\xi)\prod_{j=1}^{k}\frac{1}{m_\sigma(\xi_j)}\Bigg|\lesssim \sigma^\theta|\xi_1|^{\theta-1}|\xi_2|.
\end{equation}
\end{lemma}
\begin{proof} 
Since $m_\sigma(\xi)=m(\sigma\xi)$, by homogeneity, it suffices to consider $\sigma =1$. We split the proof in several cases.\\
\textbf{Case 1.} $|\xi_1|\leq 1$. Notice that $m(\xi_1)=1$.\\
\textbf{Subcase 1.1.} $|\xi|\leq 1$. We get $m(\xi)=m(\xi_j)=1$ for all $j$ and the left hand side of \eqref{eq:estm} is equal to zero.\\
\textbf{Subcase 1.2.} $|\xi|> 1$. We have $\frac{1}{k}<|\xi_1|\leq 1$ and using the mean value theorem, there exists $x$ such that $\frac{1}{k}<|\xi_1|\leq|x|\leq|\xi|\leq k$ and
\begin{equation*}
\left| m(\xi)-m(\xi_1)\right|\lesssim \left|m'(x)\right||\xi-\xi_1|\lesssim |\xi-\xi_1|.
\end{equation*}
Consequently,
\begin{equation*}
\begin{split}
\left|1-\frac{m(\xi)}{m(\xi_1)}\right|&=\frac{1}{m(\xi_1)}\left| m(\xi)-m(\xi_1)\right|\lesssim |\xi-\xi_1|\lesssim |\xi_2|\lesssim |\xi_1|^{\theta-1} |\xi_2|.
\end{split}
\end{equation*}
\textbf{Case 2.} $|\xi_1|>1$ and $|\xi_2|\leq 1$. Following the same steps used in Subcase 1.2, we have to estimate $|m'(x)|$ for $x$ between $\xi$ and $\xi_1$. For this, we have
 \begin{equation}
     |m'(x)|\leq \max\{|m'(\xi)|,|m'(\xi_1)|\}=m'(\widetilde{\xi}), \quad \text{where } \, |\widetilde{\xi}|=\max\{|\xi|, |\xi_1|\}.
 \end{equation}
If $0<|\widetilde{\xi}|\leq 2$ then $m'(\widetilde{\xi})$ is bounded. If $|\widetilde{\xi}|> 2$, we have $|\widetilde{\xi}|\leq |\xi_1|+(k-1)$ and then
\begin{equation*}
|m'(\widetilde{\xi})|=e^{|\xi|}\leq e^{k-1}e^{|\xi_1|}\lesssim m(\xi_1).
\end{equation*}
\textbf{Case 3.} $|\xi_2|> 1$. In this case, it suffices to show
\begin{equation}
\Bigg|m(\xi)\prod_{j=1}^{k}\frac{1}{m(\xi_j)}\Bigg|\lesssim 1.
\end{equation}
If $|\xi|\leq 2$ is obvious since $m(\xi)$ is bounded in $[-2,2]$. For $|\xi|>2$, we have
\begin{equation*}
\begin{split}
\Bigg|m(\xi)\prod_{j=1}^{k}\frac{1}{m(\xi_j)}\Bigg|&= \Bigg|e^{|\xi|}\prod_{j=1}^{k}\frac{1}{m(\xi_j)}\Bigg| \lesssim \Bigg|\frac{e^{|\xi_1|}}{m(\xi_1)}\cdots \frac{e^{|\xi_k|}}{m(\xi_k)} \Bigg| \lesssim 1.
\end{split}
\end{equation*}

\end{proof}


\section{Radius of analyticity for the generalized KdV equation}\label{sec:gKdV}

\subsection{Local well-posedness in the Gevrey space}\label{subsec:kdv}

The first step is to derive a  local well-posedness over $G^{\sigma,1}(\R)$. For the gKdV equation, this has been shown by the first author and Panthee.

\begin{theorem}[{\cite[Theorem 1.1]{BalPan}}]\label{teo:localkdv}
    Let $\sigma_0>0$ and $k\geq 4$. For given $u_0 \in G^{\sigma_0,\, 1}(\R)$, there exists a time
    \begin{equation}\label{eq:T0kdv}
    T_0=T_0(\|u_0\|_{G^{\sigma,\,1}})=\frac{c_0}{(1+\|u_0\|_{G^{\sigma_0,1}}^2)^a} \quad c_0>0, \, a>1,
\end{equation}      such that the IVP \eqref{eq:gKdV} admits a unique solution $u \in C([-T_0,\,T_0]; G^{\sigma,\, 1}(\R))\cap X_{T_0}^{\sigma_0,1,b}$ satisfying
\begin{equation}\label{eq:estukdv}
\|u\|_{X_{T_0}^{\sigma_0,1,b}}\leq C\|u_0\|_{G^{\sigma_0,\, 1}}.
\end{equation}
\end{theorem}

\subsection{Almost conserved quantity}
In order to control the radius of analyticity over an arbitrary time interval $[-T, \, T]$, we introduce an almost conserved quantity associated to the \eqref{eq:gKdV} equation. Taking into consideration the conserved quantities \eqref{eq:masskdv} and \eqref{eq:energykdv}, we define
\begin{equation}\label{eq:almost conservated}
E_\sigma (t)= \int (m_\sigma(D_x)u)^2 dx + \int (\partial_x (m_\sigma(D_x)u))^2dx + \frac{2}{(k+1)(k+2)}\int (m_\sigma(D_x)u)^{k+2}dx.
\end{equation}

Note that, for $\sigma=0$, \eqref{eq:almost conservated} turns out to be the sum of the conserved quantities \eqref{eq:masskdv} and \eqref{eq:energykdv}. However, for $\sigma>0$, this quantity fails to be conserved and the goal is to find a growth  estimate for $E_\sigma (t)$.   For this purpose,  let $U=m_\sigma(D_x)u$ and start by differentiating $E_\sigma (t)$ with respect to $t$, 

\begin{equation}\label{eq:partial E}
        \frac{d}{dt}(E_\sigma(t))=2\int U \partial_tUdx+2\int \partial_x U \partial_x(\partial_t U) dx + \frac{2}{k+1}\int U^{k+1}\partial_t U dx.
    \end{equation}
    Applying the operator $m_\sigma(D_x)$ to the \eqref{eq:gKdV} equation, we get
    \begin{equation}\label{eq:op in gkdv}
        \partial_t U +\partial_x^3 U - U^k\partial_x U=F(U),
        \end{equation}
        where $F(U)$ is given by
\begin{equation}\label{eq:F}
F(U)=-\frac{1 }{k+1}\partial_x\left[U^{k+1}-m_\sigma(D_x)\big(m_\sigma^{-1} (D_x)U)^{k+1}\big)\right].
\end{equation}       
 Using \eqref{eq:op in gkdv} in each term of \eqref{eq:partial E}, one has
\begin{equation*}
    \begin{split}
            \int U\partial_tUdx&=-\int U\partial_x^3Udx+ \frac{1}{k+2}\int\partial_x(U^{k+2})dx+\int UF(U)dx,\\
            \int \partial_xU \partial_x(\partial_tU)dx&=-\int\partial_xU\partial_x^4Udx+\int \partial_xU\partial_x(U^k\partial_xU)dx+\int \partial_xU\partial_x(F(U))dx,\\
            \int U^{k+1}\partial_t Udx&=-\int U^{k+1}\partial_x^3Udx+\int U^{2k+1}\partial_xUdx+\int U^{k+1}F(U)dx.
    \end{split}
\end{equation*}
We can assume that $U$ and all its partial derivatives tend to zero as $|x| \rightarrow \infty$ (see \cite{SS} for a detailed argument) and it follows from integration by parts that
\begin{align}
        \int U\partial_tUdx&=\int UF(U)dx, \label{eq:1a}\\
        \int \partial_xU \partial_x(\partial_tU)dx &=\frac{1}{k+1}\int U^{k+1}\partial_x^3U dx+\int \partial_xU\partial_x(F(U))dx,\label{eq:1b}\\
        \int U^{k+1}\partial_t Udx&=-\int U^{k+1}\partial_x^3Udx+\int U^{k+1}F(U)dx \label{eq:1c}.
\end{align}
Now, inserting \eqref{eq:1a}, \eqref{eq:1b} and \eqref{eq:1c} in \eqref{eq:partial E}, we get
\begin{equation}\label{eq:partial E new}
    \frac{d}{dt}E_\sigma(t)=2\int UF(U)dx + 2\int \partial_xU\partial_x(F(U))dx + \frac{2}{k+1}\int U^{k+1}F(U)dx.
\end{equation}
Integrating \eqref{eq:partial E new} in time over $[0, t']$ for $0<t'\leq T$, we obtain
\begin{equation}\label{eq:estimate E}
    E_\sigma(t') =E_\sigma(0)+R_\sigma(t'),
\end{equation}
where
\begin{equation}\label{eq:R}
\begin{split}
    R_\sigma(t')=& \,\, 2\iint \chi_{[0,\,t']}UF(U)dxdt + 2\iint \chi_{[0,\,t']}\partial_xU\partial_x(F(U))dxdt \,\, +\\
    & + \frac{2}{k+1}\iint \chi_{[0,\,t']}U^{k+1}F(U)dxdt.
    \end{split} 
\end{equation}

Now, we have to estimate each term of $|R_\sigma(t')|$ for all $0<t'\leq T$. For this purpose, we use the Cauchy-Schwarz inequality and Lemma \ref{lemma:restriction} and obtain that, for any $b=\frac{1}{2}+\epsilon$, $0<\epsilon \ll 1$, there exists $C>0$ such that
\begin{equation}\label{eq:estimate first term}
    \begin{split}
    \left|\iint \chi_{[0,\,t']}UF(U)dxdt\right| &\leq \|\chi_{[0,\,t']}U\|_{X^{1,1-b}}\|\chi_{[0,\,t']}F(U)\|_{X^{-1,b-1}}\\
    &\leq C\|U\|_{X_T^{1,b}}\|F(U)\|_{X_T^{-1,b-1}},
    \end{split}
\end{equation}

\begin{equation}\label{eq:estimate second term}
    \begin{split}
    \left|\iint \chi_{[0,\,t']}\partial_xU\partial_x(F(U))dxdt\right|& \leq \|\chi_{[0,\,t']}\partial_xU\|_{X^{0, \, 1-b}}\|\chi_{[0,\,t']}\partial_x(F(U))\|_{X^{0, \, b-1}}\\
    & \leq C\|U\|_{X_T^{1, \, 1-b}}\|\partial_x(F(U))\|_{X_T^{0, \, b-1}},
    \end{split}
\end{equation}
and 
\begin{equation}\label{eq:estimate third term}
    \begin{split}
    \left|\iint \chi_{[0,\,t']}U^{k+1}F(U)dxdt\right| &\leq \|\chi_{[0,\,t']}U\|_{X^{1,1-b}}\|\chi_{[0,\,t']}U^kF(U)\|_{X^{-1,b-1}}\\
    &\leq C\|U\|_{X_T^{1,1-b}}\|U^kF(U)\|_{X_T^{-1,b-1}}
    \end{split}
\end{equation}
for all $0<t'\leq T$.

In sequel, we find estimates for the norms involving $F(U)$ in \eqref{eq:estimate first term}-\eqref{eq:estimate third term}. To that end, observe that
\begin{equation}\label{eq:transfF}
|\widehat{F(U)}(\xi, \tau)| \lesssim|\xi|\int_{*} \Bigg|1-m_\sigma(\xi)\prod_{j=1}^{k}\frac{1}{m_\sigma(\xi_j)}\Bigg| |\widehat{U}(\xi_1,\tau_1)| \cdots |\widehat{U}(\xi_{k+1},\tau_{k+1})|,
\end{equation}
where $*$ denotes the integral over the set $\xi=\xi_1 + \cdots +\xi_{k+1}$ and $\tau=\tau_1 + \cdots +\tau_{k+1}$.
 
Without loss of generality, we can assume that $|\xi_{k+1}| \leq \cdots \leq |\xi_{1}|$. Consequently, by Lemma \ref{lemma:m},

\begin{equation}\label{eq:tranfFnova}
|\widehat{F(U)}(\xi, \tau)| \lesssim\sigma^{\theta}|\xi|\int_{*} |\xi_1|^{\theta-1}|\xi_2\|\widehat{U}(\xi_1,\tau_1)| \cdots |\widehat{U}(\xi_{k+1},\tau_{k+1})|
\end{equation}
for any $\theta\geq 1$.

Now, setting $\widehat{w}(\xi, \, \tau)=|\widehat{U}(\xi,\tau)|$ and using \eqref{eq:tranfFnova},
\begin{equation}\label{eq:normF}
\begin{split}
  \|F(U)\|_{X^{-1,b-1}}  &\lesssim \sigma^{\theta} \left\| \langle\xi\rangle^{-1}\langle \tau-\xi^3\rangle^{b-1} |\xi|\int_{*} |\xi_1|^{\theta-1}|\xi_2| |\widehat{U}(\xi_1,\tau_1)| \cdots |\widehat{U}(\xi_{k+1},\tau_{k+1})|\right\|_{L_{\xi}^2 L_{\tau}^2}\\
    &\sim \sigma^{\theta} \|\partial_x\big(|D_x|^{\theta-1}w\cdot |D_x| w\cdot w^{k-1}\big)\|_{X^{-1,b-1}}.
\end{split}
\end{equation}
Moreover,
\begin{equation}\label{eq:normpartialF}
\begin{split}
  \|\partial_xF(U)\|_{X^{0,b-1}}  &\lesssim \sigma^{\theta} \left\|\langle \tau-\xi^3\rangle^{b-1} |\xi|^2\int_{*} |\xi_1|^{\theta-1}|\xi_2| |\widehat{U}(\xi_1,\tau_1)| \cdots |\widehat{U}(\xi_{k+1},\tau_{k+1})|\right\|_{L_{\xi}^2 L_{\tau}^2}\\
    &=\sigma^{\theta} \|\partial_x\big(|D_x|^{\theta-1}w\cdot |D_x| w\cdot w^{k-1}\big)\|_{X^{1,b-1}}.
\end{split}
\end{equation}
For the norm in \eqref{eq:estimate third term}, we have that
\begin{equation}\label{eq:U4F}
\begin{split}
&|\widehat{U^kF(U)}(\xi.\tau)|\\
&\leq \int |\widehat{U^k}(\xi-\widetilde{\xi},\tau-\widetilde{\tau})|\cdot|\widehat{F(U)}(\widetilde{\xi},\widetilde{\tau})|d\widetilde{\xi}d\widetilde{\tau}\\
&\lesssim \sigma^{\theta}\int |\widehat{U}|\ast\cdots \ast|\widehat{U}| (\xi-\widetilde{\xi},\tau-\widetilde{\tau})\cdot|\widetilde{\xi}|\left(\int_{\ast}|\xi_1|^{\theta-1}|\xi_2| |\widehat{U}(\xi_1,\tau_1)| \cdots |\widehat{U}(\xi_{k+1},\tau_{k+1})|\right)d\widetilde{\xi}d\widetilde{\tau}\\
&\lesssim \sigma^{\theta}\int |\widehat{U}|\ast\cdots\ast|\widehat{U}| (\xi-\widetilde{\xi},\tau-\widetilde{\tau})\cdot\left(\int_{\ast}|\xi_1|^{\theta}|\xi_2| |\widehat{U}(\xi_1,\tau_1)| \cdots |\widehat{U}(\xi_{k+1},\tau_{k+1})|\right)d\widetilde{\xi}d\widetilde{\tau}\\
&= \sigma^{\theta}\int \big(w^k \big)^\wedge(\xi-\widetilde{\xi},\tau-\widetilde{\tau})\cdot\big(|D_x|^{\theta}w\cdot |D_x| w\cdot  w^{k-1}\big)^\wedge(\widetilde{\xi},\widetilde{\tau})d\widetilde{\xi}d\widetilde{\tau}\\
&=\sigma^{\theta}\big(|D_x|^{\theta}w\cdot |D_x| w\cdot w^{2k-1}\big)^\wedge(\xi,\tau).
\end{split}
\end{equation}
From \eqref{eq:U4F}, we find that
\begin{equation}
\|U^kF(U)\|_{X^{-1,b-1}}\lesssim\sigma^{\theta} \\|D_x|^{\theta}w\cdot |D_x| w\cdot w^{2k-1}\|_{X^{-1,b-1}}.
\end{equation}

Consequently, we aim to prove the nonlinear smoothing estimates
\begin{equation}\label{eq:aim1}
\|\partial_x\big(|D_x|^{\theta-1}w\cdot |D_x| w\cdot w^{k-1}\big)\|_{X^{1,b-1}}\lesssim \|w\|_{X^{1,b}}^{k+1}
\end{equation}
and
\begin{equation}\label{eq:aim2}
\||D_x|^{\theta}w\cdot |D_x| w\cdot w^{2k-1}\|_{X^{-1,b-1}}\lesssim \|w\|_{X^{1,b}}^{2k+1}.
\end{equation}
This is a consequence of the following result, whose proof will be postponed to Section \ref{sec:multilinear_kdv}.
\begin{lemma}\label{lem:multi_kdv}
    For any $k\ge 3$ and $\theta\in [1,2)$, the estimates
    \begin{equation}\label{eq:aim1_mult}
\left\|\partial_x\big(|D_x|^{\theta-1}w_1\cdot |D_x| w_2\cdot \prod_{j=3}^{k+1}w_j\big)\right\|_{X^{1,b-1}}\lesssim \prod_{j=1}^{k+1}\|w_j\|_{X^{1,b}}
\end{equation}
and
\begin{equation}\label{eq:aim2_mult}
\left\||D_x|^{\theta}w_1\cdot |D_x| w_2\cdot \prod_{j=3}^{2k+1}w_j\right\|_{X^{-1,b-1}}\lesssim \prod_{j=1}^{2k+1}\|w_j\|_{X^{1,b}}^{2k+1}
\end{equation}
hold.
\end{lemma}

As a consequence, we prove that the quantity $E_{\sigma}(t)$ defined in \eqref{eq:almost conservated} is almost conserved.


\begin{corollary}\label{cor: almost conserved quantity}
Let $\sigma>0$,  $k$ be even, and $\theta \in [1,2)$. Then there exists $C>0$ such that the local solution $u \in X_T^{ \sigma, \, 1, \, b}$ to the IVP \eqref{eq:gKdV} given by Theorem \ref{teo:localkdv} satisfies
\begin{equation}\label{eq:estimate energy}
    \sup_{t\in [0,\,T]} E_\sigma (t) \leq E_\sigma (0) + C\sigma^{\theta} E_\sigma (0)^{\frac{k}{2}+1}(1+E_\sigma (0)^{\frac{k}{2}}).
\end{equation}
\end{corollary}
\begin{proof}
By \eqref{eq:estimate first term}-\eqref{eq:estimate third term} and Lemma \ref{lem:multi_kdv},
\begin{equation}\label{eq:conserved quantity}
    \sup_{t\in [0,\,T]} E_\sigma (t) \leq E_\sigma (0) + C\sigma^{\theta}\|u\|_{X_T^{\sigma, \,1, \, b}}^{k+2}(1+\|u\|_{X_T^{\sigma, \,1, \, b}}^{k}).
\end{equation}
    Since $k$ is even, \eqref{eq:almost conservated} implies
    \begin{equation}\label{eq:conserved defocusing}
    E_\sigma (0)= \|u_0\|_{G^{\sigma, \, 1}}^2+\frac{2}{(k+1)(k+2)}\|m_\sigma (D_x)u_0\|_{L^{k+2}_x}^{k+2} \geq \|u_0\|_{G^{\sigma, \, 1}}^2.
\end{equation}
Together with \eqref{eq:estukdv} and \eqref{eq:conserved defocusing}, we get
\begin{equation}\label{eq:estimate u energy}
    \|u\|_{X_T^{\sigma, \,1, \, b}}\leq C E_\sigma (0)^{\frac{1}{2}}.
\end{equation}
The required estimate now follows from \eqref{eq:conserved quantity} and  \eqref{eq:estimate u energy}.
\end{proof}

\subsection{Global result and evolution of the radius of analyticity} Once we are in possession of the almost conservation law \eqref{eq:estimate energy}, we can prove Theorem \ref{teo:globalkdv}. The argument is by now standard (see, for example, \cite{AKS, BalPan, SelDaSilva, SelTes}), we present the proof for the sake of completeness.

\begin{proof}[Proof of Theorem \ref{teo:globalkdv}]

Fix $\sigma_0>0$, $k\in 2\N$, $k\geq 4$,  and $u_0 \in G^{\sigma_0, \, 1}(\R)$.
By invariance of the gKdV equation under reflection $(t,\,x) \rightarrow (-t,\,-x)$, it suffices to consider $t \geq 0$. Thus, we have to prove that, for $\theta\in [1,2)$, the local solution $u$ given by the Theorem \ref{teo:localkdv} can be extended to any time interval $[0,\,T]$ and satisfies
\begin{equation*}
    u \in C([0,\,T]:G^{\sigma(T),1}) \quad \text{for all } T>0,
\end{equation*}
where
\begin{equation}\label{eq:sigma T}
\sigma(T)\geq cT^{-\frac{1}{\theta}}
\end{equation}
and $c>0$ is a constant depending on $\|u_0\|_{G^{\sigma_0, \, 1}}, \, \sigma_0, \, k,$ and $\theta$.



For $0<\sigma\leq\sigma_0$, using Gagliardo-Nirenberg-Sobolev inequality, we have
\begin{align}
    \begin{split}
    E_{\sigma} (0)&= \|u_0\|_{G^{\sigma, \, 1}}^2+\frac{2}{(k+1)(k+2)}\|e^{\sigma |D_x|}u_0\|_{L^{k+2}}^{k+2}\\
    &\leq \|u_0\|_{G^{\sigma, \, 1}}^2+M\|D_x(m_\sigma(D_x)u_0)\|_{L^{2}}^{\frac{k}{2}}\|m_\sigma(D_x)u_0\|_{L^{2}}^{\frac{k}{2}+2}\\
    &\leq \|u_0\|_{G^{\sigma, \,1}}^2+M\|u_0\|_{G^{\sigma,\,1}}^{k+2}\label{eq:gn}\\
    &\leq \|u_0\|_{G^{\sigma_0, \,1}}^2+M\|u_0\|_{G^{\sigma_0,\,1}}^{k+2}\\
    &\leq M\left(E_{\sigma_0} (0)+E_{\sigma_0} (0)^{\frac{k}{2}+1}\right)
    \end{split}
\end{align}
for some constant $M\geq 1$. In particular, by \eqref{eq:conserved defocusing}, we have 
\begin{equation}\label{eq:control_u_0}
\|u_0\|_{G^{\sigma,\,1}}^2 \leq    E_\sigma(0)\le 2M\left(E_{\sigma_0} (0)+E_{\sigma_0} (0)^{\frac{k}{2}+1}\right).
\end{equation}
Having \eqref{eq:control_u_0}, Theorem \ref{teo:localkdv} implies that  the local lifespan in $G^{\sigma,1}(\R)$ is larger than
\begin{equation}\label{eq:delta}
    \delta=\frac{c_0}{\left(1+2M\left(E_{\sigma_0} (0)+E_{\sigma_0} (0)^{\frac{k}{2}+1}\right)\right)^a}.
\end{equation}

Let $T\geq \delta$. To cover the interval of time $[0,T]$, will repeatedly use  Theorem \ref{teo:localkdv} and Corollary \ref{cor: almost conserved quantity} with the time step $\delta$. To that end, we must ensure that estimate \eqref{eq:gn} holds for each step.

Take $n \in \N$ such that $T \in [n\delta,\,(n+1)\delta)$. We will show, by induction, that, for $j \in \{1, \cdots,\, n\}$,
\begin{equation}\label{eq:induction 1}
    \sup_{t \in [0,\, j\delta]} E_\sigma (t) \leq E_{\sigma}(0)+CM^{k+1}j\sigma^{\alpha} \left(E_{\sigma_0} (0)+E_{\sigma_0} (0)^{\frac{k}{2}+1}\right)^{\frac{k}{2}+1}\left(1+\left(E_{\sigma_0} (0)+E_{\sigma_0} (0)^{\frac{k}{2}+1}\right)^{\frac{k}{2}}\right)
\end{equation}
and
\begin{equation}\label{eq:induction 2}
    \sup_{t \in [0,\, j\delta]} E_\sigma (t) \leq 2M\left(E_{\sigma_0} (0)+E_{\sigma_0} (0)^{\frac{k}{2}+1}\right), 
\end{equation}
under the smallness conditions
\begin{equation}\label{eq: condition 1}
    \sigma \leq \sigma_0
\end{equation}
and
\begin{equation}\label{eq: condition 2}
    M^{k}\frac{T}{\delta}C\sigma^{\theta}\left(E_{\sigma_0} (0)+E_{\sigma_0} (0)^{\frac{k}{2}+1}\right)^\frac{k}{2}\left(1+\left(E_{\sigma_0} (0)+E_{\sigma_0} (0)^{\frac{k}{2}+1}\right)^\frac{k}{2}\right) \leq 1,
\end{equation}
where $C>0$ is the constant in Corollary \ref{cor: almost conserved quantity}.

In the first step, we cover the interval $[0,\, \delta]$ and, by Corollary \ref{cor: almost conserved quantity} and \eqref{eq:gn},
\begin{align}
     \sup_{t\in [0,\,\delta]} E_\sigma (t) &\leq E_\sigma (0) + C\sigma^{\theta} E_\sigma (0)^{\frac{k}{2}+1}\left(1+E_\sigma (0)^{\frac{k}{2}}\right)\label{eq:indiction base}\\
    &\leq E_\sigma (0) + C\sigma^{\theta} M^{k+1}\left(E_{\sigma_0} (0)+E_{\sigma_0} (0)^{\frac{k}{2}+1}\right)^{\frac{k}{2}+1}\left(1+\left(E_{\sigma_0} (0)+E_{\sigma_0} (0)^{\frac{k}{2}+1}\right)^{\frac{k}{2}}\right).\nonumber
\end{align}
Using \eqref{eq: condition 2}, we conclude that
\begin{equation*}
    \begin{split}
    \sup_{t \in [0,\, \delta]} E_\sigma (t)  \leq 2M\left(E_{\sigma_0} (0)+E_{\sigma_0} (0)^{\frac{k}{2}+1}\right).
    \end{split}
\end{equation*}
Now, assume that \eqref{eq:induction 2} holds for all $j \in \{1,\cdots,\, n-1\}$. For $j+1$, applying the local well posedness result from Theorem \ref{teo:localkdv} with initial data $u(j\delta)$, the almost conservation law \eqref{eq:estimate energy} and the estimates \eqref{eq:induction 1}-\eqref{eq:induction 2}, we have
\begin{align}
    & \quad\sup_{t \in [j\delta, (j+1)\delta]}E_\sigma (t)\nonumber\\&\leq E_\sigma(j\delta) + C\sigma^{\theta} E_\sigma (j\delta)^{\frac{k}{2}+1}(1+E_\sigma (j\delta)^{\frac{k}{2}})\label{eq:induction passage}\\
&\leq E_\sigma(j\delta) + CM^{k+1}\sigma^{\theta} \left(E_{\sigma_0} (0)+E_{\sigma_0} (0)^{\frac{k}{2}+1}\right)^{\frac{k}{2}+1}\left(1+\left(E_{\sigma_0} (0)+E_{\sigma_0} (0)^{\frac{k}{2}+1}\right)^{\frac{k}{2}}\right)\nonumber\\
&\leq E_\sigma(0) + CM^{k+1}(j+1)\sigma^{\theta} \left(E_{\sigma_0} (0)+E_{\sigma_0} (0)^{\frac{k}{2}+1}\right)^{\frac{k}{2}+1}\left(1+\left(E_{\sigma_0} (0)+E_{\sigma_0} (0)^{\frac{k}{2}+1}\right)^{\frac{k}{2}}\right)\nonumber
\end{align}
Moreover, since
\begin{equation*}
    j+1 \leq n \leq \frac{T}{\delta} ,
\end{equation*}
it follows from the condition \eqref{eq: condition 2} and \eqref{eq:induction passage} that
\begin{equation*}
    \begin{split}
  \sup_{t \in [j\delta,\,(j+1)\delta]}E_\sigma(t) \leq 2 M\left(E_{\sigma_0} (0)+E_{\sigma_0} (0)^{\frac{k}{2}+1}\right).
    \end{split}
\end{equation*}

Thus, we proved the local solution can be extended to the interval $[0,T]$ as long as \eqref{eq: condition 1}-\eqref{eq: condition 2} hold. These conditions are satisfied for 
\begin{equation*}
    \sigma=\min\left\{\sigma_0, \ \left[\frac{\delta}{M^{k}TC\left(E_{\sigma_0} (0)+E_{\sigma_0} (0)^{\frac{k}{2}+1}\right)^\frac{k}{2}\left(1+\left(E_{\sigma_0} (0)+E_{\sigma_0} (0)^{\frac{k}{2}+1}\right)^\frac{k}{2}\right)}\right]^{\frac{1}{\theta}}\right\}.
\end{equation*}
In particular, we conclude that the uniform radius of analyticity satisfies
$$
\sigma(T)\ge \sigma \gtrsim T^{-\frac 1\theta}.
$$
The proof is concluded by taking $\theta=2^-.$
\end{proof}

\section{Radius of analyticity for the nonlinear Schrödinger equation}\label{sec:NLS}

We now move our focus to the estimate on the lower bounds for the radius of analyticity for the nonlinear Schrödinger equation. The procedure is completely parallel to that for the generalized KdV equation.

\subsection{Local well-posedness in the Gevrey space}\label{subsec:nls}
For the \eqref{eq:schrodinger}, the local well-posedness in the Gevrey class $G^{\sigma,1}(\R)$ has been shown by Ahn, Kim and Seo.
\begin{theorem}[{\cite[Theorem 4.1]{AKS}}]\label{teo:localsch}
    Let $p\geq 3$ be odd. Given $v_0 \in G^{\sigma_0,\, 1}(\R)$, there exists a time
    \begin{equation}
    T_0=T_0(\|v_0\|_{G^{\sigma,\,1}})=\frac{c_0}{(1+\|u_0\|_{G^{\sigma_0,1}}^2)^a} \quad c_0>0, \, a>1,
\end{equation}      such that the IVP \eqref{eq:gKdV} admits a unique solution $v \in C([-T_0,\,T_0]; G^{\sigma,\, 1}(\R))\cap X_{T_0}^{\sigma_0,1,b}$ satisfying
\begin{equation}\label{eq:estusch}
\|v\|_{X_{T_0}^{\sigma_0,1,b}}\leq C\|v_0\|_{G^{\sigma_0,\, 1}}.
\end{equation}
\end{theorem}

\subsection{Almost conserved quantity} 
Taking in consideration the conserved quantities \eqref{eq:masskdv} and \eqref{eq:energysch}, we define
\begin{equation}\label{eq:almostsch}
E_\sigma (t)= \int |m_\sigma(D_x)v|^2 dx + \int |\partial_x (m_\sigma(D_x)v)|^2dx +  \frac{2}{p+1}\int |m_\sigma(D_x)v|^{p+1}dx.
\end{equation}

For $\sigma>0$, \eqref{eq:almostsch} is not conserved and our idea is to find a decay  estimate for $E_\sigma (t)$.   For this purpose,  denoting $V=m_\sigma(D_x)v$, we have
\begin{equation}\label{eq:almostconserved_nls}
E_\sigma (t')=E_\sigma (0)+\text{Im}R_\sigma(t'))
\end{equation} 
where
\begin{equation}
\begin{split}
R_\sigma(t')=\iint \chi_{[0,t']}\overline{(V-\partial_x^2V+|V|^{p-1}V)}G(V)dxdt
\end{split}
\end{equation}
and 
\begin{equation}\label{eq:G}
G(V)=-\left[|V|^{p-1}V-m_\sigma(D_x)\big(\big|m_\sigma^{-1}( D_x)V\big|^{p-1}m_\sigma^{-1}( D_x)V\big)\right].
\end{equation}

Now, we have to estimate each term of $|R_\sigma(t')|$ for all $0<t'\leq T$. For this purpose, we use the Cauchy-Schwarz inequality and Lemma \ref{lemma:restriction} and obtain that for any $b=\frac{1}{2}+\epsilon$, $0<\epsilon \ll 1$, there exists $C>0$ such that
\begin{equation}\label{eq:firstterm}
    \begin{split}
    \left|\iint \chi_{[0,\,t']}\overline{V}G(V)dxdt\right| &\leq \|\chi_{[0,\,t']}V\|_{X^{1,1-b}}\|\chi_{[0,\,t']}G(V)\|_{X^{-1,b-1}}\\
    &\leq C\|V\|_{X_T^{1,b}}\|G(V)\|_{X_T^{-1,b-1}},
    \end{split}
\end{equation}
\begin{equation}\label{eq:secondterm}
    \begin{split}
    \left|\iint \chi_{[0,\,t']}\overline{\partial_x^2V}G(V)dxdt\right|& =\left|\iint \chi_{[0,\,t']}\overline{\partial_xV}\partial_x(G(V))dxdt\right|\\
     &\leq\|\chi_{[0,\,t']}\partial_xV\|_{X^{0, \, 1-b}}\|\chi_{[0,\,t']}\partial_x(G(V))\|_{X^{0, \, b-1}}\\
    & \leq C\|V\|_{X_T^{1, \, 1-b}}\|\partial_x(G(V))\|_{X_T^{0, \, b-1}}
    \end{split}
\end{equation}
and 
\begin{equation}\label{eq:thirdterm}
    \begin{split}
    \left|\iint \chi_{[0,\,t']}\overline{|V|^{p-1}V}G(V)dxdt\right| &\leq \|\chi_{[0,\,t']}V\|_{X^{1,1-b}}\|\chi_{[0,\,t']}|V|^{p-1}G(V)\|_{X^{-1,b-1}}\\
    &\leq C\|V\|_{X_T^{1,1-b}}\||V|^{p-1}G(V)\|_{X_T^{-1,b-1}}.
    \end{split}
\end{equation}
So, it remains to find suitable bounds for the terms evolving $G(V)$. For this, observe that
    \begin{equation}\label{eq:transfG}
|\widehat{G(V)}(\xi, \tau)| \lesssim\int_{**} \Bigg|1-m_\sigma(\xi)\prod_{j=1}^{k}\frac{1}{m_\sigma(\xi_j)}\Bigg| |\widehat{V}(\xi_1,\tau_1)| \cdots |\widehat{V}(\xi_{p},\tau_{p})|,
\end{equation}
where $**$ denotes the integral over the set $\xi=\xi_1 -\xi_2+ \cdots +\xi_{p-2}-\xi_{p-1}+\xi_{p}$ and\break $\tau=\tau_1 -\tau_2+ \cdots +\tau_{p-2}-\tau_{p-1}+\tau_{p}$.
 
For the sake of simplicity, assume that $|\xi_{p}| \leq \cdots \leq |\xi_{1}|$. Consequently, from Lemma \ref{lemma:m}, we have

\begin{equation}\label{eq:tranfGnova}
|\widehat{G(V)}(\xi, \tau)| \lesssim\sigma^{\theta}\int_{**} |\xi_1|^{\theta-1}|\xi_2| |\widehat{V}(\xi_1,\tau_1)| \cdots |\widehat{V}(\xi_{p},\tau_{p})|.
\end{equation}
In particular, writing $\widehat{w}(\xi, \, \tau)=|\widehat{V}(\xi,\tau)|$, 
\begin{equation}\label{eq:normG}
\begin{split}
  \|G(V)\|_{X^{-1,b-1}}  &\lesssim \sigma^{\theta} \left\|\langle\xi\rangle^{-1}\langle \tau-\xi^3\rangle^{b-1} \int_{**} |\xi_1|^{\theta-1}|\xi_2| |\widehat{V}(\xi_1,\tau_1)| \cdots |\widehat{V}(\xi_{p},\tau_{p})|\right\|_{L_{\xi}^2 L_{\tau}^2}\\
    &\sim \sigma^{\theta} \||D_x|^{\theta-1}w\cdot\overline{ |D_x| w}\cdot |w|^{p-3}w\|_{X^{-1,b-1}}
\end{split}
\end{equation}
and
\begin{equation}\label{eq:normpartialG}
\begin{split}
  \|\partial_x G(V)\|_{X^{0,b-1}}  &\lesssim \sigma^{\theta} \left\|\langle \tau-\xi^3\rangle^{b-1}|\xi| \int_{**} |\xi_1|^{\theta-1}|\xi_2| |\widehat{V}(\xi_1,\tau_1)| \cdots |\widehat{V}(\xi_{p},\tau_{p})|\right\|_{L_{\xi}^2 L_{\tau}^2}\\
    &\sim \sigma^{\theta} \||D_x|^{\theta-1}w\cdot \overline{|D_x| w}\cdot |w|^{p-3}w\|_{X^{1,b-1}}.
\end{split}
\end{equation}
In order to control the norm in \eqref{eq:thirdterm}, we observe that
\begin{equation}\label{eq:UG}
\begin{split}
&|\widehat{|V|^{p-1}G(V)}(\xi.\tau)|\\
&\leq \int |\widehat{|V|^{p-1}}(\xi-\widetilde{\xi},\tau-\widetilde{\tau})|\cdot|\widehat{G(V)}(\widetilde{\xi},\widetilde{\tau})|d\widetilde{\xi}d\widetilde{\tau}\\
&\lesssim \sigma^{\theta}\int |\widehat{V}|\ast|\widehat{\overline{V}}|\ast\cdots \ast|\widehat{V}|\ast|\widehat{\overline{V}}| (\xi-\widetilde{\xi},\tau-\widetilde{\tau})|\left(\int_{**}|\xi_1|^{\theta-1}|\xi_2| |\widehat{V}(\xi_1,\tau_1)| \cdots |\widehat{V}(\xi_{p},\tau_{p})|\right)d\widetilde{\xi}d\widetilde{\tau}\\
&\sim \sigma^{\theta}\int \big(|w|^{p-1} \big)^\wedge(\xi-\widetilde{\xi},\tau-\widetilde{\tau})\cdot\big(|D_x|^{\theta-1}w\cdot \overline{|D_x| w}\cdot | w|^{p-3}w\big)^\wedge(\widetilde{\xi},\widetilde{\tau})d\widetilde{\xi}d\widetilde{\tau}\\
&\sim\sigma^{\theta}\big(|D_x|^{\theta-1}w\cdot \overline{|D_x| w}\cdot |w|^{2(p-2)}w\big)^\wedge(\xi,\tau).
\end{split}
\end{equation}
Together with \eqref{eq:UG}, this implies
\begin{equation}\label{eq:normVpG}
\||V|^{p-1}G(V)\|_{X^{-1,b-1}}\lesssim\sigma^{\theta} \\|D_x|^{\theta-1}w\cdot \overline{|D_x| w}\cdot |w|^{2(p-2)}w\|_{X^{-1,b-1}}.
\end{equation}

Consequently, our aim is to prove the nonlinear smoothing estimates
\begin{equation}\label{eq:aim1s}
\||D_x|^{\theta-1}w\cdot \overline{|D_x| w}\cdot |w|^{p-3}w\|_{X^{1,b-1}}\lesssim \|w\|_{X^{1,b}}^{p}
\end{equation}
and
\begin{equation}\label{eq:aim2s}
\||D_x|^{\theta-1}w\cdot\overline{ |D_x| w}\cdot |w|^{2(p-2)}w\|_{X^{-1,b-1}}\lesssim \|w\|_{X^{1,b}}^{2p-1}.
\end{equation}
\begin{remark}
    It is important to note that the two largest frequencies may correspond to different combinations of complex conjugates. As such, we actually need to prove \eqref{eq:aim1s}-\eqref{eq:aim2s} with all the possible configurations of complex conjugates.
\end{remark}

  These estimates are a particular case of the next lemma. For the proof, see Section \ref{sec:multilinear_nls}.
\begin{lemma}\label{lem:multi_nls}
    For $p=2k+1$, $k\in\N$, and $\theta\in [1,2)$, the estimates
    \begin{equation}\label{eq:aim1s_multi}
\left\||D_x|^{\theta-1}w_1\cdot \overline{|D_x| w_2}\cdot \prod_{j=3,\ j\text{ odd}}^{2k+1} w_j \cdot \prod_{j=3,\ j\text{ even}}^{2k+1} \overline{w_j} \right\|_{X^{1,b-1}}\lesssim \prod_{j=1}^{p}\|w_j\|_{X^{1,b}}
\end{equation}
and
\begin{equation}\label{eq:aim2s_multi}
\left\||D_x|^{\theta-1}w_1\cdot\overline{ |D_x| w_2}\cdot  \prod_{j=3,\ j\text{ odd}}^{4k+1} w_j \cdot \prod_{j=3,\ j\text{ even}}^{4k+1} \overline{w_j}\right\|_{X^{-1,b-1}}\lesssim \prod_{j=1}^{2p-1}\|w_j\|_{X^{1,b}}
\end{equation}
hold. The result is also valid for all other configurations of complex conjugates.
\end{lemma}

Consequently, we prove that the quantity $E_{\sigma}(t)$ defined in \eqref{eq:almostsch} is almost conserved.


\begin{corollary}\label{cor: almostconservedsch}
Let $\sigma>0$, $p\geq 3$ be odd, and $\theta \in [1,2)$. Then there exists $C>0$ such that the solution $u \in X_T^{ \sigma, \, 1, \, b}$ to the IVP \eqref{eq:schrodinger} given by Theorem \ref{teo:localsch}, we have
\begin{equation}\label{eq:energysch_sigma}
    \sup_{t\in [0,\,T]} E_\sigma (t) \leq E_\sigma (0) + C\sigma^{\theta} E_\sigma (0)^{\frac{p}{2}+1}(1+E_\sigma (0)^{\frac{p}{2}}).
\end{equation}
\end{corollary}
\begin{proof}
The application of Lemma \ref{lem:multi_nls}, together with \eqref{eq:almostconserved_nls}, \eqref{eq:normG}, \eqref{eq:normpartialG} and \eqref{eq:normVpG}, yields
\begin{equation}\label{eq:conservedsch}
    \sup_{t\in [0,\,T]} E_\sigma (t) \leq E_\sigma (0) + C\sigma^{\theta}\|u\|_{X_T^{\sigma, \,1, \, b}}^{p+1}(1+\|u\|_{X_T^{\sigma, \,1, \, b}}^{p-1}).
\end{equation}
    From \eqref{eq:almostsch}, we have
    \begin{equation}\label{eq:conserved}
    E_\sigma (0)= \|v_0\|_{G^{\sigma, \, 1}}^2+\frac{2}{p+1}\|m_\sigma(D_x)v_0\|_{L^{p+1}_x}^{p+1} \geq \|v_0\|_{G^{\sigma, \, 1}}^2.
\end{equation}
From \eqref{eq:estusch} and \eqref{eq:conserved}, we get
\begin{equation}\label{eq:estimatenergysch}
    \|v\|_{X_T^{\sigma, \,1, \, b}}\leq C E_\sigma (0)^{\frac{1}{2}}.
\end{equation}
The required estimate follows from \eqref{eq:conservedsch} and \eqref{eq:estimatenergysch}.
\end{proof}

\begin{proof}[Proof of Theorem \ref{teo:globalsch}]
This follows from Corollary \ref{cor: almostconservedsch} by applying the exact same steps used in the proof of Theorem \ref{teo:globalkdv}.

\end{proof}

\section{Proof of the multilinear estimates}\label{sec:multilinear}

As shown in Sections \ref{sec:gKdV} and \ref{sec:NLS}, the improved lower bounds on the radius of analyticity reduce to the nonlinear smoothing estimates \eqref{eq:aim1_mult}-\eqref{eq:aim2_mult} and \eqref{eq:aim1s_multi}-\eqref{eq:aim2s_multi}. In this section, we prove the validity of these estimates, which then concludes the proofs of Theorems \ref{teo:globalkdv} and \ref{teo:globalsch}.

\subsection{Reduction to frequency-restricted estimates}


To derive the required multilinear estimates, we adopt the techniques developed by the second author, Oliveira and Silva in \cite{COS}. To present the main results that are relevant to our work, let us consider a generic dispersive PDE
\begin{eqnarray}\label{eq:genericPDE}
\begin{cases}
\partial_t u-iL(D)u=N(u), \\
u(x,0)=u_0(x)\in H^{s}(\R^d),
\end{cases}
\end{eqnarray}
where $D=\partial_x/i$, $L(D)$ is a linear differential operator in the spatial variables given by a real Fourier symbol $L(\xi)$, i.e.,
\begin{equation*}
L(D)f(x):=(L(\xi)\widehat{f}(\xi))^\vee(x)
\end{equation*}
and $N(u)$ is a general nonlinearity defined by the multilinear form $N(u)=N(u,\dots, u)$ given by
\begin{equation*}
\begin{split}
\widehat{N(u_1,\dots,u_k)}(\xi)&=\int_{\xi=\xi_1+\cdots+\xi_k}m(\xi_1,\dots,\xi_k)\left(\prod_{j=1}^{k'}\widehat{u}_j(\xi_j)\right)\left(\prod_{j=k'+1}^{k}\widehat{\overline{u}}_j(\xi_j)\right)d\xi_1\cdots d\xi_{k-1}.
\end{split}
\end{equation*}	

We also define the corresponding phase function as
\begin{equation*}
\Phi(\xi_1,\dots,\xi_k):= L(\xi)-\sum_{j=1}^{k'}L(\xi_j)+\sum_{j=k'+1}^k L(\xi_j)
\end{equation*}
and the convolution hyperplane
\begin{gather*}
\Gamma_\xi=\left\lbrace (\xi_1,\cdots,\xi_k)\in (\R^d)^k: \xi=\sum_{j=1}^{k'}\xi_j-\sum_{j=k'+1}^k \xi_j\right\rbrace.
\end{gather*}


For simplicity, we write $\xi=\xi_0$ and $\tau=\tau_0$. Given a set of indices $A\subset \{0,\dots,k\}$, we abbreviate $"\xi_j, \, j\in A"$ as $"\xi_{j\in A}"$ and $d\xi_{j\in A}$ and $d\xi_{j\notin A}$ denotes the integration over all $\xi_j$ in $A$ and not in $A$, respectively. 

\begin{lemma}[{\cite[Lemma 2]{COS}}]\label{lemma:freqrestr}
Given $s\in \R$ and $\epsilon>0$, suppose that there exist $\emptyset\neq A \subsetneq \{0,\dots,k\}$, $\mathcal{M}_j=\mathcal{M}_j(\xi_1,\dots,\xi_k)\geq 0$, $j=1,2$, such that
\begin{equation*}
(\mathcal{M}_1\mathcal{M}_2)^\frac{1}{2}=\frac{|m(\xi_1,\dots,\xi_k)|\langle\xi\rangle^{s+\epsilon}}{\prod_{j=1}^k\langle\xi_j\rangle^s}
\end{equation*}
and, for any $M>1$,
\begin{equation*}
\sup_{\xi_{j\in A},\alpha} \int_{\Gamma_\xi} \mathcal{M}_1\mathbbm{1}_{|\Phi-\alpha|<M}d\xi_{j\notin A} +\sup_{\xi_{j\notin A},\alpha} \int_{\Gamma_\xi} \mathcal{M}_2\mathbbm{1}_{|\Phi-\alpha|<M}d\xi_{j\in A} \lesssim M^\beta.
\end{equation*}
Then, for all $-1/2<b-1<b'<0$ with $2b'<-\beta$, the estimate
\begin{equation}\label{eq:estimateN}
\|N(u_1,\dots,u_k)\|_{X^{s+\epsilon,b'}}\lesssim \prod_{j=1}^k\|u_j\|_{X^{s,b}}
\end{equation} holds.
\end{lemma}

Recently, in \cite{CorLei}, the second author and Leite explored the relationship between frequency-restricted estimates for nonlinear dispersive equations over $\R^d$  with the same linear operator but involving nonlinearities of different degrees. For this purpose, we must consider the convolution hyperplane with deformation $\sigma$,
\begin{equation*}
\Gamma_\xi^\sigma=\left\lbrace (\xi_1,\cdots,\xi_k)\in (\R^d)^k: \xi=\sum_{j=1}^{k'}\xi_j-\sum_{j=k'+1}^k \xi_j+\sigma, \, \, 	|\sigma|\lesssim \min_{j\in\{1,\dots,k\}}|\xi_j|\right\rbrace,
\end{equation*}

Now, given $s>s_c(k,d)=$ and $\epsilon < \epsilon_c(k,d,s)$, we say that a \textit{flexible frequency-restricted estimate} holds if there exist $\emptyset \neq A \subsetneq \{0,\dots,k\}$, pairs $(s_j, s_j'),\ j \in \{0,\dots,k\}$, and $ (\epsilon_0, \epsilon_0') $ such that
\[
\left\{
\begin{aligned}
s_j + s_j' &= 2s, \quad j \in \{0,\dots,k\}, \\
\epsilon_0 + \epsilon_0' &= 2\epsilon,
\end{aligned}
\right.
\]
and, for any $ M \geq 1 $,
\begin{equation}\label{eq:ffre1}
\sup_{\sigma, \alpha, \xi_{j \notin A}} \int_{\Gamma^{\sigma}_\xi} \frac{|m_k(\xi_1, \ldots, \xi_k)| \langle \xi \rangle^{s_0 + \epsilon_0}}{\prod_{j=1}^k \langle \xi_j \rangle^{s_j}} \mathbbm{1}_{|\Phi - \alpha| < M} \, d\xi_{j\in A} \lesssim M^{1^-}
\end{equation}
and
\begin{equation}\label{eq:ffre2}
\sup_{\sigma, \alpha, \xi_{j \in A}} \int_{\Gamma^{\sigma}_\xi} \frac{|m_k(\xi_1, \ldots, \xi_k)| \langle \xi \rangle^{s_0' + \epsilon_0'}}{\prod_{j=1}^k \langle \xi_j \rangle^{s_j'}} \mathbbm{1}_{|\Phi - \alpha| < M} \, d\xi_{j\notin A} \lesssim M^{1^-}.
\end{equation}

Moreover, given $ k_0 \leq k $, we say that $\Phi_{k_0} : (\mathbb{R}^d)^{k_0 + 1} \to \mathbb{R}$ is a $k_0$-descent of $\Phi$ if there exists a set 
$ J \subset \{ 1,\dots, k \}$, $ |J| = k - k_0$, such that
\begin{equation}
\Phi_{k_0} \cong \left. \Phi \right|_{H_J}, \quad \text{where } H_J = \{ (\xi, \xi_1, \ldots, \xi_k) \in (\mathbb{R}^d)^{k+1} : \xi_j = 0,\ \forall j \in J \}.
\end{equation}

Let be $s_{\text{LWP}}$ denote the optimal regularity threshold for which the local well-posedness holds for all $s>s_{\text{LWP}}$. For such $s	$, define
\begin{equation*}
\epsilon_c(k,d,s):=\min\{(k-1)(s-s_{\text{LWP}}),l-n-1\},
\end{equation*} 
where $k$ is the order of the nonlinearity, $l$ is the dispersion and $n$ is the derivative loss associated with the initial value problem.

\begin{theorem}[Induction in $k$, \cite{CorLei}]\label{teo:indk}
Fix \(k \geq 3\). Let $k_0 < k$ such that 
\begin{equation*}
|m_k(\xi_1,\dots,\xi_k)|\lesssim |m_{k_0}(\widetilde{\xi_1},\dots,\widetilde{\xi_{k_0}})|
\end{equation*} 
where $(\widetilde{\xi}_1, \dots, \widetilde{\xi}_{k_0})$ is the vector consisting in the $k_0$ largest (up to multiplicative constants) components of $(\xi_1, \ldots, \xi_k)$.
Assume further that $s_{\text{LWP}}=s_c(k_0,d)$ and that for every  $s > s_c(k_0, d)$, $\epsilon < \epsilon_c(k_0, d)$ and for all $k_0$-descents of $\Phi$, a flexible frequency-restricted estimate holds. Then, given $s > s_c(k, d)$ and $\epsilon < \epsilon_c(k,d,s)$, the hypotheses of Lemma \ref{lemma:freqrestr} are satisfied.
\end{theorem}

\remark\label{remark:avoid}{(\cite[Remark 2.2]{CorLei})} The conclusion remains valid if, in the definition of the frequency-restricted estimate, the set $\Gamma_\xi^\sigma$ is replaced with
\begin{equation*}
\Gamma_\xi^\sigma(c)=\Gamma_\xi^\sigma\cap \{(\xi_1,\dots,\xi_{k_0})\in(\R^d)^{k_0}:|\xi_1|\sim\cdots\sim |\xi_{k_0}|\Rightarrow \sigma \not\simeq c\xi\} ,
\end{equation*}
for any $c\in \R$ such that
\begin{equation*}
c\neq 1-\frac{k_0}{K}, \quad k_0\leq K\leq k.
\end{equation*}

As a consequence of the proof of Theorem \ref{teo:indk}, one can also prove an induction result for regularities $s>\frac d2$. This is particularly well suited for our purposes, since the multilinear estimates \eqref{eq:aim1_mult}-\eqref{eq:aim2_mult} and \eqref{eq:aim1s_multi}-\eqref{eq:aim2s_multi} are stated for $s=1$.

\begin{corollary}\label{cor:indk}
Fix $k\geq 3$ and suppose $s>\frac{d}{2}$. If there exists $k_0 < k$ such that one has
\begin{equation*}
|m_k(\xi_1,\dots,\xi_k)|\lesssim |m_{k_0}(\widetilde{\xi_1},\dots,\widetilde{\xi_{k_0}})|
\end{equation*} 
and, for all $k_0$-descents of $\Phi$, a flexible frequency-restricted estimate holds, then the conditions of Lemma \ref{lemma:freqrestr} are satisfied.
\end{corollary}

In order to obtain the frequency-restricted estimates described above, we will need the following auxiliary lemma.
\begin{lemma}[{\cite[Lemma 5]{COS}}] \label{lemma:auxiliar} For $M,N>0$ and $\alpha\in \R$ fixed,
\begin{equation}
\int_{|p|,|q|<N} \mathbbm{1}_{|p^2\pm q^2-\alpha|<M}dpdq\lesssim M^{1^-}N^{0^+}
\end{equation}
and
\begin{equation}
\int \mathbbm{1}_{|p^2-\alpha|<M}dp\lesssim M^{1/2}.
\end{equation}

\end{lemma}

The last lemma will also be useful to reduce the derivation of \eqref{eq:ffre1}-\eqref{eq:ffre2} to an almost equivalent frequency-restricted estimate with a full power of $M$.

\begin{lemma}[{\cite[Lemma 3.3]{CLS}}]\label{lem:eta1}
	Let $K:\R^n\times \R^m\times \R\to \R^+$ be a positive measurable function such that, for some $N\in \mathbb{N}$,
	$$
	|K(x,y,M)|\lesssim \max\{1,|x|,|y|\}^N,\quad \forall x\in\R^n,\ y\in \R^m, \ M\in \R.
	$$
	Suppose that
	$$
	\sup_y \int K(x,y,M)\max\{1,|x|,|y|\}^{0^+} dx \lesssim M,\quad \mbox{for all }M>1.
	$$
	Then there exists $0<\eta<1$ such that
	$$
	\sup_y \int K(x,y,M) dx \lesssim M^{\eta},\quad \mbox{for all } M>1.
	$$
\end{lemma}

\subsection{Estimates for the generalized KdV equation}\label{sec:multilinear_kdv}
The following lemma provides a proof for the multilinear estimate \eqref{eq:aim1_mult}. This estimate is akin with a nonlinear smoothing estimate with a gain of $\theta<2$ derivatives, which is more than the known nonlinear smoothing effect of $\theta<1$ derivatives shown in \cite{CorLei}. This improved smoothing is a consequence of having the extra derivatives distributed through various frequencies, instead of being concentrated in the outgoing frequency.

\begin{lemma}\label{lemma:estkdv} For $k=4$ and $\theta\in [1,2)$, the flexible frequency-restricted estimate associated with \eqref{eq:aim1_mult} holds.
\end{lemma}
\begin{proof}
We begin by ordering $|\xi_1|\geq \cdots \geq |\xi_5|$ and consider the worst-case scenario $|\xi|\ge |\xi_1|$, which implies\footnote{We say that $a\lesssim b$ if there exists a universal constant such that $|a|\le C|b|$. If $C$ can be chosen arbitrarily small, we write $a\ll b$. If $a\lesssim b$ and $b\lesssim a$, we say that $a\sim b$. Finally, we write $a\simeq b$ if $|a-b|\ll \max\{|a|,|b|\}$.} that $|\xi_1|\sim |\xi|$. As the case $|\xi|<1$ is trivial, we suppose that $|\xi|>1$. Consequently,
\begin{equation}
\frac{|\xi|\langle\xi\rangle|\xi_1|^{\theta-1}|\xi_2| }{\langle \xi_1\rangle\langle \xi_2\rangle\langle \xi_3\rangle\langle \xi_4\rangle\langle \xi_5\rangle}\lesssim \frac{|\xi|^{\theta}}{\langle \xi_3\rangle\langle \xi_4\rangle\langle \xi_5\rangle}
\end{equation}
\textbf{Case 0. $|\xi_5|\gtrsim |\xi_1|$.}
As all frequencies are comparable, we may bound
$$
\frac{|\xi|\langle\xi\rangle|\xi_1|^{\theta-1}|\xi_2| }{\langle \xi_1\rangle\langle \xi_2\rangle\langle \xi_3\rangle\langle \xi_4\rangle\langle \xi_5\rangle}\lesssim \frac{|\xi|^{\theta}}{\langle \xi_3\rangle\langle \xi_4\rangle\langle \xi_5\rangle} \lesssim \frac{|\xi|\jap{\xi}^{\frac34 + \epsilon}}{\prod_{j=1}^5 \jap{\xi_j}^{\frac34}}, \quad \epsilon=\theta-1.
$$
This reduces to the flexible frequency-restricted estimate shown in \cite[Proposition 3.1]{CorLei} for $\epsilon<1$.

\bigskip

Henceforth, we focus on the case $|\xi_5|\ll |\xi_1|$. This implies that $|\xi_4|\not\simeq |\xi|$, otherwise
$$
\xi\simeq\xi_1+\xi_2+\xi_3+\xi_4\simeq k\xi, \quad k\text{ even},
$$
which is absurd.

By Lemma \ref{lem:eta1}, it suffices to derive the frequency-restricted estimates with a full power of $M$. We split the proof in two cases.

\noindent
\textbf{Case 1.} $|\xi_3|\gtrsim |\xi_2|$.  We show that 
\begin{equation}\label{eq:kdv_smooth_1}
\sup_{\sigma,\alpha,\xi,\xi_2,\xi_4}\int_{\Gamma_\xi^\sigma}\frac{|\xi|^{\theta}}{\langle\xi_3\rangle\langle\xi_5\rangle^2}\mathbbm{1}_{|\Phi-\alpha|<M}d\xi_1d\xi_3\lesssim M^{1^-}
\end{equation}
and
\begin{equation}
\sup_{\sigma,\alpha,\xi_1,\xi_3,\xi_5}\int_{\Gamma_\xi^\sigma}\frac{|\xi|^{\theta}}{\jap{\xi_3}\langle\xi_4\rangle^2}\mathbbm{1}_{|\Phi-\alpha|<M}d\xi d\xi_2\lesssim M^{1^-}.
\end{equation}
For the first estimate, let us fix $\xi,\xi_2,\xi_4$. Observe that
\begin{equation}
\left|\frac{\partial\Phi}{\partial\xi_1}\right|=3|\xi_1^2-\xi_5^2|\gtrsim |\xi_1|^2\sim|\xi|^2
\end{equation}
and so we can perform the change of variable $\xi_1\mapsto \Phi$:
\begin{align}
\int_{\Gamma_\xi^\sigma}\frac{|\xi|^{\theta+0^+}}{\langle\xi_3\rangle\langle\xi_5\rangle^2}\mathbbm{1}_{|\Phi-\alpha|<M}d\xi_1d\xi_3 &\lesssim \int_{\Gamma_\xi^\sigma}\frac{|\xi|^{\theta-2+0^+}}{\langle\xi_3\rangle}\mathbbm{1}_{|\Phi-\alpha|<M}d\Phi d\xi_3\\
&\lesssim |\xi|^{\theta-2+0^+}M\int_{|\xi_3|\le|\xi|}\frac{1}{\langle\xi_3\rangle}d\xi_3\\
&\lesssim M |\xi|^{\theta-2+0^+}\ln |\xi| \lesssim M.
\end{align}
for $\theta\in[1,2)$. Applying Lemma \ref{lem:eta1}, we find \eqref{eq:kdv_smooth_1}. The second estimate follows from the same steps, using the fact that $|\xi_3|\gtrsim |\xi_2|$.

\noindent
\textbf{Case 2.} $|\xi_3|\ll |\xi_2|$. We have $|\xi_2|\not\simeq |\xi|$, since otherwise $\xi\simeq \xi_1+\xi_2 \simeq k\xi$, $k$ even.

We then prove that
\begin{equation}
\sup_{\sigma,\alpha,\xi,\xi_2,\xi_4}\int_{\Gamma_\xi^\sigma}\frac{|\xi|^{\theta}}{\langle\xi_3\rangle^2\langle\xi_5\rangle^2}\mathbbm{1}_{|\Phi-\alpha|<M}d\xi_1d\xi_3\lesssim M^{1^-}
\end{equation}
and
\begin{equation}\label{eq:kdv_smooth_2}
\sup_{\sigma,\alpha,\xi_1,\xi_3,\xi_5}\int_{\Gamma_\xi^\sigma}\frac{|\xi|^{\theta}}{\langle \xi_4\rangle^2}\mathbbm{1}_{|\Phi-\alpha|<M}d\xi d\xi_4\lesssim M^{1^-}.
\end{equation}
The first estimate follows exactly as in Case 1. For the second, let us fix $\xi_1,\xi_3,\xi_5$ and observe that
\begin{equation}
\left|\frac{\partial\Phi}{\partial\xi}\right|=3|\xi^2-\xi_2^2|\gtrsim |\xi|^2\sim|\xi_1|^2.
\end{equation}
In particular, we may perform the change of variables $\xi\mapsto \Phi$:
\begin{equation}
    \int_{\Gamma_\xi^\sigma}\frac{|\xi|^{\theta+0^+}}{\langle \xi_4\rangle^2}\mathbbm{1}_{|\Phi-\alpha|<M}d\xi d\xi_4 \lesssim \int_{\Gamma_\xi^\sigma}\frac{|\xi_1|^{\theta-2+0^+}}{\langle\xi_4\rangle^2}\mathbbm{1}_{|\Phi-\alpha|<M}d\Phi d\xi_4\lesssim M.
\end{equation}
Estimate \eqref{eq:kdv_smooth_2} is now a consequence of Lemma \ref{lem:eta1}.

\end{proof}

For \eqref{eq:aim2_mult}, we do not need to use any sort of smoothing effect. Heuristically, when compared with \eqref{eq:aim1_mult}, we see that there are two derivatives missing and thus, instead of having to gain $\theta<2$ derivatives, we do not need to gain any derivatives at all.

\begin{lemma}\label{lemma:estkdv} For $k=2$ and $\theta\in [1,2)$, the flexible frequency-restricted estimate associated with \eqref{eq:aim2_mult} holds.
\end{lemma}
\begin{proof}
The multiplier associated with \eqref{eq:aim2_mult} is 
\begin{equation}
\mathcal{M}=\frac{|\xi_1|^{\theta}|\xi_2| }{\langle\xi\rangle\prod_{j=1}^{5}\langle\xi_j\rangle}.
\end{equation}
where $\xi=\xi_1+\cdots+\xi_5$. In the worst-case scenario $|\xi_1|\sim |\xi_2|\geq \cdots\geq |\xi_5|\geq |\xi|$, the estimate is trivial if $|\xi_1|<1$. For $|\xi_1|>1$,
\begin{equation}\label{eq:aim2mult}
\frac{|\xi_1|^{\theta}|\xi_2| }{\langle\xi\rangle\prod_{j=1}^{5}\langle\xi_j\rangle}\lesssim \frac{|\xi_1|\jap{\xi_1}}{\langle\xi\rangle\prod_{j=2}^{5}\langle\xi_j\rangle}.
\end{equation}
This corresponds to the classical multilinear estimate
$$
\left\| \partial_x(u_1\dots u_5) \right\|_{X^{1,b'}}\lesssim \prod_{j=1}^5\|u_j\|_{X^{1,b}},
$$
for which the flexible frequency-restricted estimate was shown in \cite[Proposition 3.1]{CorLei}.
\end{proof}

\subsection{Multilinear estimates for the nonlinear Schrödinger equation}\label{sec:multilinear_nls} In this section. we prove the multilinear estimates \eqref{eq:aim1s_multi} and \eqref{eq:aim2s_multi}.

\begin{lemma}\label{lemma:estsch} For $p=3$ and $\theta\in[1,2)$, the flexible-restricted estimate associated with \eqref{eq:aim1s_multi} holds.
\end{lemma}
\begin{proof}
We begin by ordering $|\xi|\geq|\xi_1|\geq |\xi_2| \geq |\xi_3|$ which implies $|\xi_1|\sim |\xi|$. For $|\xi|<1$, the multiplier is bounded and the integration over bounded sets is controlled directly. For $|\xi|>1$, we bound the associated multiplier as
\begin{equation}\label{eq:est_multi_nls}
\frac{\langle\xi\rangle|\xi_1|^{\theta-1}|\xi_2| }{\langle \xi_1\rangle\langle \xi_2\rangle\langle \xi_3\rangle}\lesssim \frac{|\xi|^{\theta-1}}{\langle \xi_3\rangle}.
\end{equation}

The 3-descents of the resonance function are of the form
\begin{equation}
\Phi=\xi^2+\lambda_1\xi_1^2+\lambda_2\xi_2^2+\lambda_3\xi_3^2, \quad \lambda_j\in \{\pm 1\}, \quad j=1,2,3.
\end{equation}
We want to estimate
\begin{equation}\label{eq:est1s}
\sup_{\sigma,\alpha,\xi,\xi_2}\int_{\Gamma_\xi^\sigma}\frac{|\xi|^{\theta-1}}{\langle\xi_3\rangle}\mathbbm{1}_{|\Phi-\alpha|<M}d\xi_1\lesssim M^{1^{-}}
\end{equation}
and
\begin{equation}\label{eq:est2s}
\sup_{\sigma,\alpha,\xi_1,\xi_3}\int_{\Gamma_\xi^\sigma}\frac{|\xi|^{\theta-1}}{\langle\xi_3\rangle}\mathbbm{1}_{|\Phi-\alpha|<M}d\xi \lesssim M^{1^{-}},
\end{equation}
where
\begin{equation}
\Gamma_{\xi}^{\sigma}=\{(\xi_1,\xi_2,\xi_3)\in \R^3: \xi+\sigma=\xi_1-\xi_2+\xi_3, \quad \min_{j}|\xi_j|\gtrsim |\sigma|\}.
\end{equation}

\noindent\textbf{Case 1.} $|\xi_3|\ll |\xi|$. We claim that $|\xi|\not\simeq |\xi_2|$. Indeed, if $\xi \simeq \pm \xi_2$, then
\begin{equation}
\xi \simeq \xi_1+\xi_2 \simeq k\xi \qquad \text{for some} \quad k\in\{0,\pm 2\},
\end{equation}
which is impossible. 
For the first estimate \eqref{eq:est1s}, we observe that
\begin{equation}
\left|\frac{\partial\Phi}{\partial \xi_1}\right|=|2\lambda_1\xi_1-2\lambda_3\xi_3|\gtrsim |\xi_1|\sim |\xi|.
\end{equation}
So we can perform the change of variables $\xi_1\mapsto \Phi$ obtaining
\begin{equation}
\begin{split}
\int_{\Gamma_\xi^\sigma}\frac{|\xi|^{\theta-1+0^+}}{\langle\xi_3\rangle}\mathbbm{1}_{|\Phi-\alpha|<M}d\xi_1&\lesssim \int_{\Gamma_\xi^\sigma}\frac{|\xi|^{\theta-2+0^+}}{\langle\xi_3\rangle}\mathbbm{1}_{|\Phi-\alpha|<M}d\Phi\\
& \lesssim M
\end{split}
\end{equation}
for $\theta\in [1,2)$. \eqref{eq:est1s} now follows from Lemma \ref{lem:eta1}.

For the estimate \eqref{eq:est2s}, we have that
\begin{equation}
\left|\frac{\partial\Phi}{\partial \xi}\right|=|2\xi-2\lambda_2\xi_2|\gtrsim |\xi|\sim |\xi_1|
\end{equation}
and we can perform the change of variables $\xi\mapsto \Phi$ analogous to the previous estimate.
\medskip

\noindent\textbf{Case 2.} $|\xi_3|\gtrsim |\xi_1|$. In this case, all frequencies are comparable and we may make suitable permutations whenever necessary. To prove the flexible frequency-restricted estimates, it suffices to check that we are either in a nonstationary case (and then we proceed as in the previous case) or in a nondegenerate stationary point. In the second scenario, for the estimate \eqref{eq:est1s}, for $\xi,\xi_2$ fixed, we can write 
\begin{equation}
\Phi=\xi^2(1+\lambda_1p_1^2+\lambda_2p_2^2+\lambda_3p_3^2)=\xi^2P(p_1),
\end{equation} 
where $p_j=\frac{\xi_j}{\xi}$ for $j\in\{1,2,3\}$ and $p_1+p_2+p_3=1$. We are supposing that $|p_1-p_1^0|\ll 1$ with $\frac{\partial P}{\partial p_1}(p_1^0)=0$ and $\frac{\partial^2 P}{\partial p_1^2}(p_1^0)\neq 0$ so that we can use the Morse lemma in a neighbourhood of $p_1^0$ and we find that there exists a change of coordinates $p_1\mapsto q_1$ such that
\begin{equation}
P(p_1)=P(p_1^0)+\frac{1}{2}\frac{\partial^2P}{\partial p_1^2}(p_1^0)q_1^2
\end{equation}
and we obtain
\begin{equation}
\begin{split}
\int_{\Gamma_\xi^\sigma}\frac{|\xi|^{\theta-1}}{\langle\xi_3\rangle}\mathbbm{1}_{|\Phi-\alpha|<M}d\xi_1&\lesssim \int_{|p_1-p_1^0|\ll 1} |\xi|^{\theta-2}|\xi|\mathbbm{1}_{|\xi^2P-\alpha|<M}dp_1\\
&\lesssim \int_{|q_1|\ll 1} |\xi|^{\theta-1}\mathbbm{1}_{\left|\xi^2\left(P(p_1^0)+\frac{1}{2}\frac{\partial^2P}{\partial p_1^2}q_1^2\right)-\alpha\right|<M}dq_1\\
&\lesssim \int_{|q_1|\ll 1} |\xi|^{\theta-1}\mathbbm{1}_{\left|\left(P(p_1^0)+\frac{1}{2}\frac{\partial^2P}{\partial p_1^2}q_1^2\right)-\frac{\alpha}{\xi^2}\right|<\frac{M}{\xi^2}}dq_1\\
&\lesssim |\xi|^{\theta-1}\sqrt{\frac{M}{\xi^2}}\\
& \lesssim M^{1^-}
\end{split}
\end{equation}
for $\theta\in [1,2)$.

Consequently, all we have to prove in that, up to a possible rearrangements between $\xi_1,\xi_2,\xi_3$ there are no degenerate stationary critical points. We will divide the proof of this according to the different combinations of signs of $\lambda_1, \lambda_2$ and $\lambda_3$.\\
\textbf{Subcase 2.1.} $(+,+,+)$. For $\xi, \xi_2$ fixed, $\frac{\partial^2\Phi}{\partial \xi_1^2}=4$ and for $\xi_1,\xi_3$ fixed, $\frac{\partial^2\Phi}{\partial \xi^2}=4$.\\
\textbf{Subcase 2.2.} $(+,+,-)$. For $\xi_1,\xi_3$ fixed, $\frac{\partial^2\Phi}{\partial \xi^2}=4$. For $\xi, \xi_2$ fixed, $\frac{\partial^2\Phi}{\partial \xi_1^2}=0$ and $\frac{\partial\Phi}{\partial \xi_1}=2\xi_1+2\xi_3$. In this case, we have to consider the region $\xi_1\simeq -\xi_3$.

\textbf{Subcase 2.2.1.} $\xi_2\not\simeq \xi_3$. It suffices to exchange $\xi_1$ and $\xi_2$: for $\xi_2,\xi_3$ fixed, $\frac{\partial^2\Phi}{\partial \xi^2}=4$ and, for $\xi,\xi_1$ fixed, $\frac{\partial\Phi}{\partial \xi_2}=2\xi_2-2\xi_3\neq 0$.

\textbf{Subcase 2.2.2.} $\xi_2\simeq \xi_3$. In this case, $\xi_1\simeq -\xi_2\simeq -\xi_3$ and we exchange  $\xi_2$ and $\xi_3$. With this, for $\xi, \xi_3$ fixed, $\frac{\partial^2\Phi}{\partial \xi_3^2}=4$. Moreover, for $\xi_1,\xi_2$ fixed, $\frac{\partial\Phi}{\partial \xi}=2\xi-2\xi_3$. If $\left|\frac{\partial\Phi}{\partial \xi}\right|\ll |\xi|$, then $\xi\simeq \xi_3$ and
\begin{equation}
\xi+\sigma=\xi_1-\xi_2+\xi_3\simeq -\xi-\xi+\xi=-\xi, \quad \text{this is} \quad \sigma\simeq -2\xi
\end{equation}
which is an avoidable value of $\sigma$ (cf. Remark \ref{remark:avoid}).\\
\textbf{Subcase 2.3.} $(-,+,-)$. For $\xi, \xi_2$ fixed, $\frac{\partial^2\Phi}{\partial \xi_1^2}=-4$ and for $\xi_1,\xi_3$ fixed, $\frac{\partial^2\Phi}{\partial \xi^2}=4$.\\
\noindent\textbf{Subcase 2.4.} $(-,-,-)$. Exchanging $\xi$ with $\xi_3$ and inverting the sign of the resonance function, this case is equivalent to Subcase 2.2. 

\end{proof}

\begin{lemma}\label{lemma:estsch} For $p=2$ and $\theta\in[1,2)$, the flexible-restricted estimate associated with \eqref{eq:aim2_mult} holds.
\end{lemma}
\begin{proof}
In this case, the associate multiplier is 
\begin{equation}
\mathcal{M}=\frac{|\xi_1|^{\theta}|\xi_2| }{\langle\xi\rangle\prod_{j=1}^{3}\langle\xi_j\rangle}.
\end{equation}
In the worst-case scenario $|\xi_1|\ge |\xi_2|\ge |\xi_3|\ge |\xi|$, we  estimate
\begin{equation}
\mathcal{M}=\frac{|\xi_1|^{\theta}|\xi_2| }{\langle\xi\rangle\prod_{j=1}^{3}\langle\xi_j\rangle}\lesssim \frac{|\xi_1|^{\theta-1}}{\jap{\xi}}.
\end{equation}
This corresponds exactly to \eqref{eq:est_multi_nls} and thus the proof follows from that of Lemma \ref{lem:multi_nls}.
\end{proof}

\bibliography{Biblio}

\begin{thebibliography}{10}

\bibitem{AKS}
J.~Ahn, J.~Kim, and I.~Seo.
\newblock On the radius of spatial analyticity for defocusing nonlinear
  {S}chrödinger equations.
\newblock {\em Discrete Contin. Dyn. Syst.}, 40(1):423--439, 2020.

\bibitem{BalPan}
M.~Baldasso and M.~Panthee.
\newblock Improved algebraic lower bound for the radius of spatial analyticity
  for the generalized {K}d{V} equation.
\newblock {\em Nonlinear Anal. Real World Appl.}, 77:Paper No. 104054, 11,
  2024.

\bibitem{BonaGruKal}
J.~L. Bona, Z.~Gruji\'c, and H.~Kalisch.
\newblock Algebraic lower bounds for the uniform radius of spatial analyticity
  for the generalized {K}d{V} equation.
\newblock {\em Ann. Inst. H. Poincar\'e{} C Anal. Non Lin\'eaire},
  22(6):783--797, 2005.

\bibitem{BPSS}
J.~L. Bona, G.~Ponce, J.~C. Saut, and C.~Sparber.
\newblock Dispersive blow-up for nonlinear {S}chrödinger equations revisited.
\newblock {\em J. Math. Pures Appl. (9)}, 102(4):782--811, 2014.

\bibitem{BonaSaut}
J.~L. Bona and J.~C. Saut.
\newblock Dispersive blowup of solutions of generalized korteweg-de vries
  equations.
\newblock {\em J. Differential Equations}, 103(1):3--57, 1993.

\bibitem{BonaSaut2}
J.~L. Bona and J.~C. Saut.
\newblock Dispersive blow-up {II}. {S}chrödinger-type equations, optical and
  oceanic rogue waves.
\newblock {\em Chinese Ann. Math. Ser. B}, 31(6):793--818, 2010.

\bibitem{Bourg1}
J.~Bourgain.
\newblock Fourier transform restriction phenomena for certain lattice subsets
  and applications to nonlinear evolution equations. {I}. {S}chrödinger
  equations.
\newblock {\em Geom. Funct. Anal.}, 3(2):107--156, 1993.

\bibitem{Bourg2}
J.~Bourgain.
\newblock Fourier transform restriction phenomena for certain lattice subsets
  and applications to nonlinear evolution equations. {II}. the
  {K}d{V}-equation.
\newblock {\em Geom. Funct. Anal.}, 3(3):209--262, 1993.

\bibitem{Bourg_nls_book}
J.~Bourgain.
\newblock {\em Global solutions of nonlinear {S}chrödinger equations},
  volume~46 of {\em American Mathematical Society Colloquium Publications}.
\newblock American Mathematical Society, Providence, RI, 1999.

\bibitem{CKSTT2}
J.~Colliander, M.~Keel, G.~Staffilani, H.~Takaoka, and T.~Tao.
\newblock Sharp global well-posedness for {K}d{V} and modified {K}d{V} on
  {$\mathbb{R}$} and {$\mathbb{T}$}.
\newblock {\em J. Amer. Math. Soc.}, 16(3):705--749, 2003.

\bibitem{CKSTT1}
J.~Colliander, M.~Keel, G.~Staffilani, H.~Takaoka, and T.~Tao.
\newblock Multilinear estimates for periodic {K}d{V} equations, and
  applications.
\newblock {\em J. Funct. Anal.}, 211(1):173--218, 2004.

\bibitem{C_4kdv}
S.~Correia.
\newblock Improved global well-posedness for the quartic {K}orteweg-de {V}ries
  equation.
\newblock {\em Proc. Amer. Math. Soc.}, 152(12):5117--5136, 2024.

\bibitem{CorLei}
S.~Correia and P.~Leite.
\newblock Sharp local existence and nonlinear smoothing for dispersive
  equations with higher-order nonlinearities.
\newblock {\em Math. Z.}, 310(3):Paper No. 60, 21, 2025.

\bibitem{CLS}
S.~Correia, F.~Linares, and J.~D. Silva.
\newblock Sharp local well-posedness for the {S}chrödinger-{K}orteweg-de
  {V}ries system.
\newblock {\em Preprint, arXiv:2408.10028}.

\bibitem{COS}
S.~Correia, F.~Oliveira, and J.~D. Silva.
\newblock Sharp local well-posedness and nonlinear smoothing for dispersive
  equations through frequency-restricted estimates.
\newblock {\em SIAM J. Math. Anal.}, 56(4):5604--5633, 2024.

\bibitem{tzirakis1}
M.~B. Erdoğan and N.~Tzirakis.
\newblock Global smoothing for the periodic {K}d{V} evolution.
\newblock {\em Int. Math. Res. Not. IMRN}, (20):4589--4614, 2013.

\bibitem{tzirakis2}
M.~B. Erdoğan and N.~Tzirakis.
\newblock Talbot effect for the cubic non-linear {S}chrödinger equation on the
  torus.
\newblock {\em Math. Res. Lett.}, 20(6):1081--1090, 2013.

\bibitem{tzirakis3}
M.~B. Erdoğan and N.~Tzirakis.
\newblock Regularity properties of the cubic nonlinear {S}chrödinger equation
  on the half line.
\newblock {\em J. Funct. Anal.}, 271(9):2539--2568, 2016.

\bibitem{Farah_5kdv}
L.~G. Farah.
\newblock Global rough solutions to the critical generalized {K}d{V} equation.
\newblock {\em J. Differential Equations}, 249(8):1968--1985, 2010.

\bibitem{FNP}
R.~O. Figueira, M.~Nogueira, and M.~Panthee.
\newblock Lower bounds on the radius of analyticity for a system of nonlinear
  quadratic interactions of the {S}chrödinger-type equations.
\newblock {\em Z. Angew. Math. Phys.}, 75(4):Paper No. 136, 14, 2024.

\bibitem{FigPan_mkdv}
R.~O. Figueira and M.~Panthee.
\newblock Improved lower bound for the radius of analyticity for the modified
  {K}d{V} equation.
\newblock {\em Preprint, arXiv:2406.08400}.

\bibitem{FigPan}
R.~O. Figueira and M.~Panthee.
\newblock Decay of the radius of spatial analyticity for the modified {K}d{V}
  equation and the nonlinear {S}chrödinger equation with third order
  dispersion.
\newblock {\em NoDEA Nonlinear Differential Equations Appl.}, 31(4):Paper No.
  68, 23, 2024.

\bibitem{BirGet}
T.~Getachew and B.~Belayneh.
\newblock New asymptotic lower bound for the radius of analyticity of solutions
  to nonlinear {S}chrödinger equation.
\newblock {\em Anal. Appl. (Singap.)}, 22(5):815--832, 2024.

\bibitem{GruKal}
Z.~Gruji\'c and H.~Kalisch.
\newblock Local well-posedness of the generalized {K}orteweg-de {V}ries
  equation in spaces of analytic functions.
\newblock {\em Differential Integral Equations}, 15(11):1325--1334, 2002.

\bibitem{GPS}
A.~Gr\"unrock, M.~Panthee, and J.~D. Silva.
\newblock A remark on global well-posedness below {$L^2$} for the g{K}d{V}-3
  equation.
\newblock {\em Differential Integral Equations}, 20(11):1229--1236, 2007.

\bibitem{Hayashi_Analyt}
N.~Hayashi.
\newblock Analyticity of solutions of the {K}orteweg-de {V}ries equation.
\newblock {\em SIAM J. Math. Anal.}, 22(6):1738--1743, 1991.

\bibitem{HuangWang}
J.~Huang and M.~Wang.
\newblock New lower bounds on the radius of spatial analyticity for the {K}d{V}
  equation.
\newblock {\em J. Differential Equations}, 266(9):5278--5317, 2019.

\bibitem{IMOS}
B.~Isom, D.~Mantzavinos, S.~Oh, and A.~Stefanov.
\newblock Polynomial bound and nonlinear smoothing for the benjamin-ono
  equation on the circle.
\newblock {\em J. Differential Equations}, 297:25--46, 2021.

\bibitem{katznelson}
Y.~Katznelson.
\newblock {\em An introduction to harmonic analysis}.
\newblock Cambridge Mathematical Library. Cambridge University Press,
  Cambridge, third edition, 2004.

\bibitem{KeraaniVargas}
S.~Keraani and A.~Vargas.
\newblock A smoothing property for the {$L^2$}-critical nls equations and an
  application to blowup theory.
\newblock {\em Ann. Inst. H. Poincar\'e{} C Anal. Non Lin\'eaire},
  26(3):745--762, 2009.

\bibitem{LinPasSil}
F.~Linares, A.~Pastor, and J.~D. Silva.
\newblock Dispersive blow-up for solutions of the {Z}akharov-{K}uznetsov
  equation.
\newblock {\em Ann. Inst. H. Poincar\'e{} C Anal. Non Lin\'eaire},
  38(2):281--300, 2021.

\bibitem{MSWX}
C.~Miao, S.~Shao, Y.~Wu, and G.~Xu.
\newblock The low regularity global solutions for the critical generalized
  {K}d{V} equation.
\newblock {\em Dyn. Partial Differ. Equ.}, 7(3):265--288, 2010.

\bibitem{Sato}
T.~Sato.
\newblock Spatial analyticity of solutions to quadratic nonlinear
  {S}chrödinger equations with mass resonance.
\newblock {\em J. Math. Anal. Appl.}, 551(2):Paper No. 129684, 14, 2025.

\bibitem{Selberg_DKG}
S.~Selberg.
\newblock On the radius of spatial analyticity for solutions of the
  {D}irac-{K}lein-{G}ordon equations in two space dimensions.
\newblock {\em Ann. Inst. H. Poincar\'e{} C Anal. Non Lin\'eaire},
  36(5):1311--1330, 2019.

\bibitem{SelDaSilva}
S.~Selberg and D.~O. da~Silva.
\newblock Lower bounds on the radius of spatial analyticity for the {K}d{V}
  equation.
\newblock {\em Ann. Henri Poincar\'e}, 18(3):1009--1023, 2017.

\bibitem{SelTes}
S.~Selberg and A.~Tesfahun.
\newblock On the radius of spatial analyticity for the quartic generalized
  {K}d{V} equation.
\newblock {\em Ann. Henri Poincar\'e}, 18(11):3553--3564, 2017.

\bibitem{tas_nls}
A.~Tesfahun.
\newblock On the radius of spatial analyticity for cubic nonlinear
  {S}chrödinger equations.
\newblock {\em J. Differential Equations}, 263(11):7496--7512, 2017.

\bibitem{Tasf_kdv}
A.~Tesfahun.
\newblock Asymptotic lower bound for the radius of spatial analyticity to
  solutions of {K}d{V} equation.
\newblock {\em Commun. Contemp. Math.}, 21(8):1850061, 33, 2019.

\end{thebibliography}
\bibliographystyle{plain}

	\normalsize

\begin{center}

	{\scshape Mikaela Baldasso}\\
{\footnotesize
	Instituto de Matemática, Estatística e Computação Científica,\\
    Departamento de Matemática,\\ 
Universidade Estadual de Campinas\\
Rua Sérgio Buarque de Holanda, 651, 13083-859, Campinas–SP, Brazil\\
	mikaelabaldasso@gmail.com
}

	\bigskip
	{\scshape Simão Correia}\\
	{\footnotesize
		Center for Mathematical Analysis, Geometry and Dynamical Systems,\\
		Department of Mathematics,\\
		Instituto Superior T\'ecnico, Universidade de Lisboa\\
		Av. Rovisco Pais, 1049-001 Lisboa, Portugal\\
		simao.f.correia@tecnico.ulisboa.pt
	}

\end{center}

\end{document}